\documentclass{amsart}
\usepackage{amsmath,amsthm,amssymb}
\usepackage{graphicx}

\newtheorem{theorem}{Theorem}[section]
\newtheorem{proposition}[theorem]{Proposition}
\newtheorem{lemma}[theorem]{Lemma}
\newtheorem{corollary}[theorem]{Corollary}

\newtheorem{conjecture}[theorem]{Conjecture}
\theoremstyle{definition}
\newtheorem{definition}[theorem]{Definition}
\newtheorem{remark}[theorem]{Remark}

\newtheorem{example}[theorem]{Example}

\theoremstyle{theorem}

\title[Corks, exotic 4-manifolds and knot concordance]{Corks, exotic 4-manifolds and knot concordance}
 
\author[Kouichi Yasui]{Kouichi Yasui}
\dedicatory{Dedicated to Professor Makoto Sakuma on the occasion of his 60th birthday}
\thanks{The author was partially supported by JSPS KAKENHI Grant Number 16K17593.}
\date{May 24, 2015. \textit{Revised}: September 28, 2017}

\address{Department of Pure and Applied Mathematics, Graduate School of Information Science and Technology, Osaka University, 
1-5 Yamadaoka, Suita, Osaka 565-0871, Japan}
\email{kyasui@ist.osaka-u.ac.jp}
\begin{document}
%\maketitle

\begin{abstract}We show that, for each integer $n$, there exist infinitely many pairs of $n$-framed knots representing homeomorphic but non-diffeomorphic (Stein) 4-manifolds, which are the simplest possible exotic 4-manifolds regarding handlebody structures. 
To produce these examples, we introduce a new description of cork twists and utilize satellite maps. As an application, we produce knots with the same $0$-surgery which are not concordant for any orientations, disproving the Akbulut-Kirby conjecture given in 1978.  
\end{abstract}

\maketitle

\section{Introduction}\label{sec:intro}A framed link in $S^3$ gives a 4-manifold by attaching 2-handles to the 4-ball $D^4$ along the framed link. Such framed link presentations of 4-manifolds (together with gauge theoretical results) have many applications to low dimensional topology (cf.\ \cite{A_book, GS, OS1}). 
In this paper, we give a method for producing framed knots representing exotic (i.e.\ pairwise homeomorphic but non-diffeomorphic) 4-manifolds. We note that these are the simplest possible exotic 4-manifolds, regarding the minimal number of handles in their handle decompositions. As an application, we disprove the Akbulut-Kirby conjecture on knot concordance. 

\subsection{Framed knots representing exotic 4-manifolds}We say that a compact oriented smooth 4-manifold is represented by a framed knot, if it is obtained from $D^4$ by attaching a 2-handle along the framed knot. Not much is known about framed knots representing exotic 4-manifolds (exotic  framed knots for short). 
The first example of a pair of exotic framed knots was constructed by Akbulut~\cite{A2} in 1991 (see also \cite{AM}). Kalm\'{a}r and Stipsicz~\cite{KS} extended this example to an infinite family of such pairs. We remark that the framings of these known examples are all $-1$ (and also $+1$ after taking the mirror images of the knots).  Furthermore, one 4-manifold of each pair admits a Stein structure, but the other does not, hence these smooth structures are easily distinguished. 

Here we give infinitely many pairs of exotic $n$-framed knots for any integer $n$, and furthermore the corresponding smooth structures cannot be distinguished by existences of Stein structures. 
\begin{theorem}\label{intro:thm:knot}For each integer $n$, there exist infinitely many distinct pairs of $n$-framed knots such that each pair represents a pair of homeomorphic but non-diffeomorphic 4-manifolds. Furthermore, each pair can be chosen so that the both 4-manifolds admit Stein structures. 
\end{theorem}

We also give other types of exotic framed knots regarding Stein structures. See Corollaries~\ref{cor:Stein_non-Stein} and \ref{both_non-Stein}. For the background of exotic Stein 4-manifolds, we refer to \cite{AY5} and the references therein. 

Our method for producing these examples utilizes satellite maps. Let us recall some definitions. For a knot $P$ in $S^1\times D^2$, let $P(K)$ denote the (untwisted) satellite of a knot $K$ in $S^3$ with the pattern $P$. %, where the tubular neighborhood $N(K)$ of $K$ is identified with $S^1\times D^2$ by the $0$-framing of $K$. 
We may regard $P$ as a map from the set of knots in $S^3$ to itself. This map $P$ is called a (untwisted) \textit{satellite map}. For an integer $n$, identifying the tubular neighborhood of $K$ with $S^1\times D^2$ by $n$-framing, we have an \textit{$n$-twisted satellite} $P_n(K)$ of $K$ with the pattern $P$.  %We call the resulting map $P_n$ an \textit{$n$-twisted satellite map of $P$}. 

We give pairs of satellite maps such that each pair transforms a knot with a certain simple condition into a pair of exotic framed knots. For the definitions of the symbols, see Section~\ref{sec:preliminary}.

\begin{theorem}\label{intro:thm:satellite_knot}There exists a pair of satellite maps $P$ and $Q$ satisfying the following: for any fixed integer $n$, if a knot $K$ in $S^3$ satisfies $2g_4(K)=\overline{ad}(K)+2$ and $n\leq \widehat{tb}(K)$, then the $n$-twisted satellite knots $P_n(K)$ and $Q_n(K)$ with $n$-framings represent 4-manifolds which are homeomorphic but non-diffeomorphic to each other. Furthermore, both of these 4-manifolds admit Stein structures if $n\leq \overline{tb}(K)-1$. 
\end{theorem}

In fact, we construct infinitely many such pairs of satellite maps (see Theorem~\ref{thm:satellite:detail}). They are probably mutually distinct pairs, but we do not pursue this point here.  
To construct these satellite maps, we introduce a new description of cork twists, which we call a \textit{hook surgery} (see Sections~\ref{sec:new_description} and \ref{sec:hook_dot}). %The hook surgery description of a cork twist is of independent interest and a key of this paper. 
This description immediately gives us a pair of satellite maps such that an exchange of the maps has an effect of a cork twist. Applying the construction of exotic (Stein) 4-manifolds obtained by Akbulut and the author in~\cite{AY2, AY5}, we then obtain the theorem above. 

We remark that, for each integer $n$, there are infinitely many knots $K$ satisfying the assumption of the above theorem (e.g.\ positive torus knots). Furthermore, if a knot $K$ satisfies the assumption, then $P_n(K)$ and $Q_n(K)$ also satisfy the assumption (see Remark~\ref{thm:torus:remark}). Therefore Theorem~\ref{intro:thm:knot} follows from the above theorem.

We further discuss satellite maps from the viewpoint of 4-manifolds. For a (untwisted) satellite map $P$, an integer $n$, and a knot $K$, let $P^{(n)}(K)$ denote the 4-manifold represented by the (untwisted) satellite knot $P(K)$ with $n$-framing. We may regard $P^{(n)}$ as a map from the set of knots in $S^3$ to the set of smooth 4-manifolds. We use the following terminologies. 
\begin{definition}We call the map $P^{(n)}$ an \textit{$n$-framed 4-dimensional satellite map}. We say that two $n$-framed  4-dimensional satellite maps $P^{(n)}$ and $Q^{(n)}$ are smoothly (resp.\ topologically) the same, if  two 4-manifolds $P^{(n)}(K)$ and $Q^{(n)}(K)$ are diffeomorphic (resp.\ homeomorphic) to each other for any knot $K$ in $S^3$. 
\end{definition}

Our satellite maps reveal the following new difference between topological and smooth categories in 4-dimensional topology. 
\begin{corollary}\label{intro:cor:4-dim satellite}For each integer $n$, there exist $n$-framed 4-dimensional satellite maps which are topologically the same but smoothly distinct. 
\end{corollary}

In addition, we also discuss a certain surgery, which we call a \textit{dot-zero surgery}. This surgery is a natural generalization of a cork twist along a Mazur type cork, and many known exotic 4-manifolds are essentially obtained by this surgery (cf.\ \cite{A_book, AY1, GS}). We give a sufficient condition on a link such that a dot-zero surgery induced from the link does not change the smooth structure of a 4-manifold (Proposition~\ref{prop:dot-zero}). Applying this result, we discuss symmetric links and corks. 

Let us recall that a symmetric link $L$ consisting of two unknotted components with the linking number one gives a compact contractible %4-dimensional handlebody by drawing a dot on one component of $L$ and $0$ on top of the other component. 
4-manifold $X_L$ and an involution $\tau_L$ on the boundary $\partial X_L$, where a symmetric link means a link admitting an involution on $S^3$ exchanging the two components (see Section~\ref{sec:hook_dot} for details). It has been well-known that the pair $(X_L, \tau_L)$ is a cork for many symmetric links $L$ (\cite{AKa12, AM, AY1}), and such a cork is called Mazur type. The only known exception is the Hopf link  to the best of the author's knowledge, and the corresponding contractible 4-manifold is $D^4$. Hence it is natural to ask whether all other symmetric links yield corks. However, applying Proposition~\ref{prop:dot-zero}, we show that there exist infinitely many symmetric links which do not yield corks. 

\begin{theorem}\label{intro:thm:noncork}There exist infinitely many distinct symmetric links $L_1,L_2,\dots$ in $S^3$ consisting of two unknotted components with the linking number one such that any $(X_{L_i}, \tau_{L_i})$ is not a cork. Furthermore each link $L_i$ can be chosen so that $X_{L_i}$ is a Stein 4-manifold not homeomorphic to the 4-ball. 
\end{theorem}

We note that any symmetric link presentation of $D^4$ does not yield a cork, since any self-diffeomorphism of $S^3$ extends to a self-diffeomorphism of $D^4$. In contrast to this fact, we also show that, even if one symmetric link presentation does not yield a cork, another symmetric link presentation of the same (Stein) 4-manifold can yield a cork. 

\begin{theorem}\label{intro:thm:cork_noncork}There exists a compact contractible $($Stein$)$ 4-manifold $C$ admitting two links $L_1$ and $L_2$ in $S^3$ satisfying the following conditions. 
\begin{itemize}
 \item Each $L_i$ is a symmetric link consisting of two unknotted components with the linking number one. 
 \item Both the 4-manifolds $X_{L_1}$ and $X_{L_2}$ are diffeomorphic to $C$. 
 \item $(X_{L_1},\tau_{L_1})$ is not a cork, but $(X_{L_2},\tau_{L_2})$ is a cork. %, where each $\varphi_{L_i}$ is the involution on the boundary $\partial C$ induced from the symmetry of $L_i$.%The involutions $\varphi_{L_1}$ and $\varphi_{L_2}$ on $\partial C$ respectively induced from the symmetries of $L_1$ and $L_2$ satisfy that $\varphi_{L_1}$ does not extend to any self-diffeomorphism of $C$, but $\varphi_{L_2}$ extends to a self-diffeomorphism of $C$. Consequently  
\end{itemize}
\end{theorem}
We remark that the above $(X_{L_2},\tau_{L_2})$ is the Mazur cork found by Akbulut~\cite{A1}, and the proof utilizes the new description of cork twists. 
%This result shows that not all symmetric link presentation of a compact contractible (Stein) 4-manifold other than $D^4$ gives a cork. 

\subsection{Application to knot concordance} 
Let us recall the definition and the background on knot concordance. Two oriented knots $S^1\to S^3$ are said to be concordant if there exists a proper smooth embedding $S^1\times [0,1]\to S^3\times [0,1]$ whose restriction to $S^1\times \{0,1\}$ are the two knots. 
In 1978, Akbulut and Kirby gave the following conjecture (see Problem 1.19 in the Kirby's problem list~\cite{K}). 

\begin{conjecture}[Akbulut and Kirby]\label{conjecture:AK}If $0$-framed surgeries on two knots give the same 3-manifold, then the knots are concordant. 
\end{conjecture}
To be precise, this conjecture is stated as follows. Note that there are oriented knots which are not concordant to the same knots with the reverse orientation (e.g.\ \cite{KL}). 

\begin{conjecture}[cf.\ \cite{AT}]\label{conjecture:AT}If $0$-framed surgeries on two knots give the same 3-manifold, then the knots with relevant orientations are concordant. 
\end{conjecture}
We note that many concordance invariants are defined via the $0$-surgery of a knot (often with the positive meridian), and two knots with the same $0$-surgery have the same concordance invariants for relevant orientations (see Problem 1.19 in \cite{K} and introduction in \cite{CP}). Furthermore, if every homotopy 4-sphere is diffeomorphic to the standard one, then this conjecture is true when one knot is slice (see \cite{K}). %Also, 3-manifolds given by 0-surgeries of concordant knots are homology cobordant (see \cite{CFHH}). 
On the other hand, Cochran, Franklin, Hedden, and Horn \cite{CFHH} produced non-concordant pairs of knots under the weaker assumption that $0$-framed surgeries on knots are homology cobordant. For related results in the link case, we refer to \cite{CP} and the references therein. Recently Abe and Tagami~\cite{AT} proved that, if the slice-ribbon conjecture is true, then the Akbulut-Kirby conjecture is false. 

In this paper we disprove the Akbulut-Kirby conjecture. In fact, we show that exotic $0$-framed knots given by Theorem~\ref{intro:thm:satellite_knot} are counterexamples. %Furthermore, we prove the same result for an arbitrary surgery coefficient. We remark that 

\begin{theorem}\label{thm:intro}There exists a pair of knots in $S^3$ with the same $0$-surgery which are not concordant for any orientations. Furthermore, there exist infinitely many distinct pairs of such knots. 
\end{theorem}

It is natural to ask whether a knot concordance invariant provides an invariant of 3-manifolds given by $0$-surgeries of knots. %, since many known knot concordance invariants are defined via the $0$-surgery on a knot as we already mentioned. 
Cochran, Franklin, Hedden, and Horn \cite{CFHH} proved that the invariants $\tau$, $s$ (\cite{OzSz, Ras}) and the 4-genus are not invariants of homology cobordism classes of such 3-manifolds. Furthermore, a result of Levine \cite{Lev} shows that the same holds for the invaiant $\epsilon$ (\cite{Hom}). Strengthening these results, our examples show 
\begin{corollary}\label{intro:cor:tau}The knot concordance invariants $\tau$, $s$, $\epsilon$ and the 4-genus are not invariants of 3-manifolds given by $0$-surgeries of knots. 
\end{corollary}

\section{Preliminaries and notations}\label{sec:preliminary}In this section, we recall basic definitions and facts and introduce our notations. Our main tools are 4-dimensional (Stein) handlebodies and Legendrian knots. For their details, we refer to \cite{A_book, GS, OS1}. 
\subsection{Knots}
Let $K$ be a knot in $S^3$. The \textit{4-genus} $g_4(K)$ is defined to be the minimal genus of a properly embedded smooth surface in $D^4$ bounded by the knot $K$.  For an integer $n$, the 4-manifold represented by $K$ with $n$-framing means the 4-manifold obtained from $D^4$ by attaching a 2-handle along $K$ with $n$-framing. The \textit{$n$-shake genus} $g_s^{(n)}(K)$ is defined to be the minimal genus of a smoothly embedded closed surface representing a generator of the second homology group of this 4-manifold (\cite{A0, A_book}). %Unlike the 4-genus, the $n$-shake genus gives an invariant of smooth 4-manifolds. 
Note the obvious relation $g_s^{(n)}(K)\leq g_4(K)$. The lemma below is well-known and obvious from the definition. 

\begin{lemma}Two concordant knots in $S^3$ have the same 4-genus. 
\end{lemma}
For an integer $n$, let $f_n: S^1\times D^2\to N(K)\subset S^3$ be a trivialization of the tubular neighborhood $N(K)$ of a knot $K$ corresponding to the $n$-framing of $K$. 
For a knot $P$ in the solid torus $S^1\times D^2$, we denote the knot $f_0(P)$ in $S^3$ by $P(K)$. The knot $P(K)$ is called the \textit{satellite} of $K$ with the pattern $P$, and the knot $f_n(K)$ in $S^3$ is  called an \textit{$n$-twisted satellite} of $K$ with the pattern $P$. The map on the set of knots in $S^3$ given by $K\mapsto P(K)$ is called a (untwisted) \textit{satellite map}. %We call the resulting map $P_n$ an \textit{$n$-twisted satellite map of $P$}. 
%%%%%%%%%%

\subsection{Legendrian knots}Throughout this paper, a Legendrian knot in $S^3$ means the one with respect to the standard tight contact structure on $S^3$. 
For a Legendrian knot $\mathcal{K}$ in $S^3$, let $tb(\mathcal{K})$ and $r(\mathcal{K})$ denote the Thurston-Bennequin number and the rotation number, respectively. According to \cite{E1, G1}, the 4-manifold represented by $\mathcal{K}$ with the framing $tb(\mathcal{K})-1$ admits a Stein structure. Since any Stein 4-manifold can be embedded into a closed minimal complex surface of general type with $b_2^+>1$ (\cite{LM1}), the well-known adjunction inequality for this closed 4-manifold (\cite{FS3, KM1,MST, OzSz_ad}) together with Gompf's Chern class formula (\cite{G1}) gives the following version for the Stein 4-manifold. Note that this holds even for the genus zero case (cf.\ \cite{GS, OS1, AY5}), unlike the version for general closed 4-manifolds. 
\begin{theorem}[\cite{AM, LM2}] $tb(\mathcal{K})-1+|r(\mathcal{K})|\leq 2g_s^{(tb(\mathcal{K})-1)}(\mathcal{K})-2$.
\end{theorem}

We define the \textit{adjunction number} $ad(\mathcal{K})$ as the left side of this inequality, namely, 
\begin{equation*}
ad(\mathcal{K})=tb(\mathcal{K})-1+|r(\mathcal{K})|. 
\end{equation*}
For a knot $K$ in $S^3$, we say that a Legendrian knot $\mathcal{K}$ is a Legendrian representative of $K$ if $\mathcal{K}$ is smoothly isotopic to $K$. We denote the set of Legendrian representatives of $K$ by $\mathcal{L}(K)$. We define the \textit{maximal Thurston-Bennequin number} $\overline{tb}(K)$ and the \textit{maximal adjunction number} $\overline{ad}(K)$ of $K$ by  
\begin{equation*}
 \overline{tb}(K)= \max\{tb(\mathcal{K})\mid \mathcal{K}\in \mathcal{L}(K)\} \quad\text{and}\quad 
\overline{ad}(K)= \max\{ad(\mathcal{K})\mid \mathcal{K}\in \mathcal{L}(K)\}.
\end{equation*}
%Note that $\overline{tb}(K)$ and $\overline{ad}(K)$ are finite numbers by the above adjunction inequality. 
We also define the symbol $\widehat{tb}(K)$ by 
\begin{equation*}
\widehat{tb}(K)=\max\{tb(\mathcal{K})\mid \mathcal{K}\in \mathcal{L}(K),\; ad(\mathcal{K})=\overline{ad}({K})\}.
\end{equation*}
%%%%%%%%%%%%%%%%%%
\subsection{Corks}For a compact contractible oriented smooth 4-manifold $C$ and an involution $\tau:\partial C \to \partial C$, the pair $(C,\tau)$ is called a \textit{cork}, if $\tau$ extends to a self-homeomorphism of $C$ but does not extend to any self-diffeomorphism of $C$.  For a smooth 4-manifold $X$ containing $C$ as a submanifold, remove $C$ from $X$ and glue it via the involution $\tau$. We call this operation a \textit{cork twist} along $(C,\tau)$. We note that any self-diffeomorphism on the boundary of a smooth contractible 4-manifold extends to a self-homeomorphism of the 4-manifold (\cite{Fr}). For examples of corks, see \cite{AY1}. Note that, unlike the original definition in \cite{AY1}, we require a cork to be contractible, since this case seems most useful for applications. Also we do not require a cork to admit a Stein structure, since we regard a cork as a useful tool for various constructions, and thus the flexible definition seems more appropriate. However, all corks used in this paper actually admit Stein structures.% Also we do not require a cork to admit a Stein structure, since the Stein condition is useful but not necessary for applications such as constructions. All corks used in this paper actually admit Stein structures, and we would like to call them Stein corks. 
%%%%%%%%%

%%%%%%%%%%%%%%%%%
%%%%%%%%%%%%%%%%
\section{A new description of cork twists}\label{sec:new_description}
In this section, we give a family of corks and introduce a new description of cork twists, which we call a hook surgery in Section~\ref{sec:hook_dot}. The new description is a key of our main results  and has independent interest.

We first give a family of corks, which was introduced by Auckly-Kim-Melvin-Ruberman \cite{AKMR}. For an integer $n$, let $V_n$ be the contractible 4-manifold given by the left diagram of Figure~\ref{fig:V_n}. It is easy to check that this manifold is diffeomorphic to the right 4-manifold by isotopy, where the box $-\frac{n}{2}$ denotes $-n$ right-handed half twists. Note that the link is symmetric, i.e., there exists an involution $S^3\to S^3$ which exchanges the components of the link. Let $g_n:\partial V_n\to \partial V_n$ be the involution obtained by first surgering $S^1\times D^3$ to $D^2\times S^2$ and then surgering the other $D^2\times S^2$ to $S^1\times D^3$ (i.e.\ replacing the dot and the zero in the diagram). In the case $n\geq 0$, by converting the 1-handle notation, we obtain the Stein handlebody presentation of $V_n$ in Figure~\ref{fig:Legendrian_V_n}, where the left-handed full twists in the box denote the Legendrian version shown in Figure~\ref{fig:Legendrian_left_twists}. Hence according to \cite{E1, G1}, $V_n$ admits a Stein structure for $n\geq 0$. We remark that $(V_0,g_0)$ is the same cork as $(W_1,f_1)$ in \cite{AY1}. As pointed out by a referee, our cork $(V_n,g_n)$ coincides with the cork $(C_h,\tau)$ $(h=-\frac{n}{2})$ introduced by Auckly-Kim-Melvin-Ruberman \cite{AKMR}.%, whose diagram is different from ours. 
\begin{figure}[h!]
\begin{center}
\includegraphics[width=3.5in]{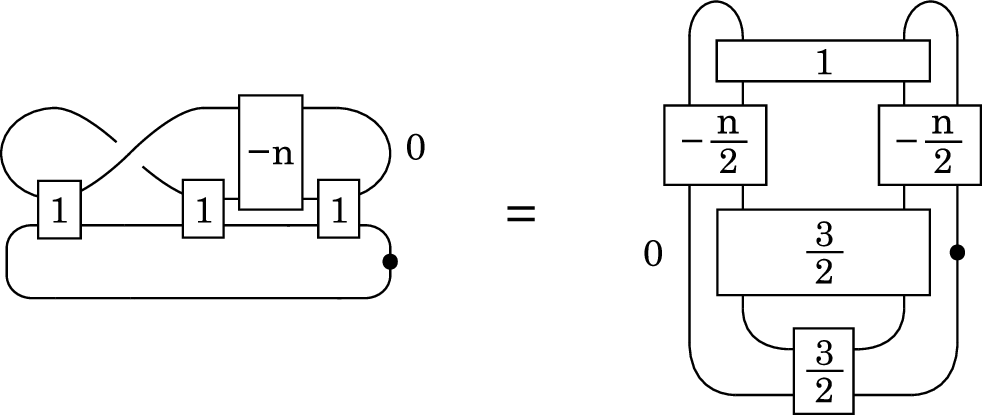}
\caption{Two diagrams of $V_n$}
\label{fig:V_n}
\end{center}
\end{figure}
\begin{figure}[h!]
\begin{center}
\includegraphics[width=3.2in]{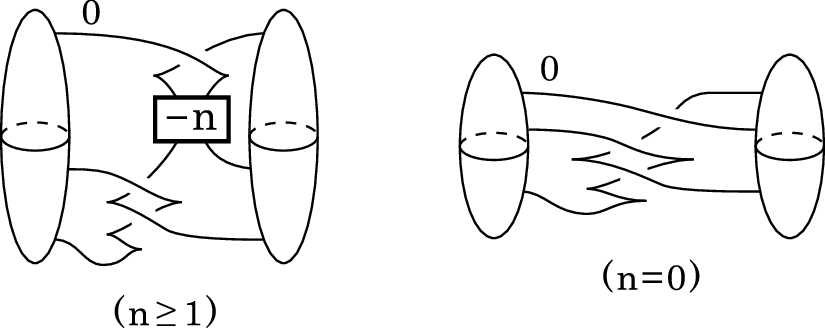}
\caption{A Stein handlebody presentation of $V_n$ $(n\geq 0)$}
\label{fig:Legendrian_V_n}
\end{center}
\end{figure}
\begin{figure}[h!]
\begin{center}
\includegraphics[width=3.8in]{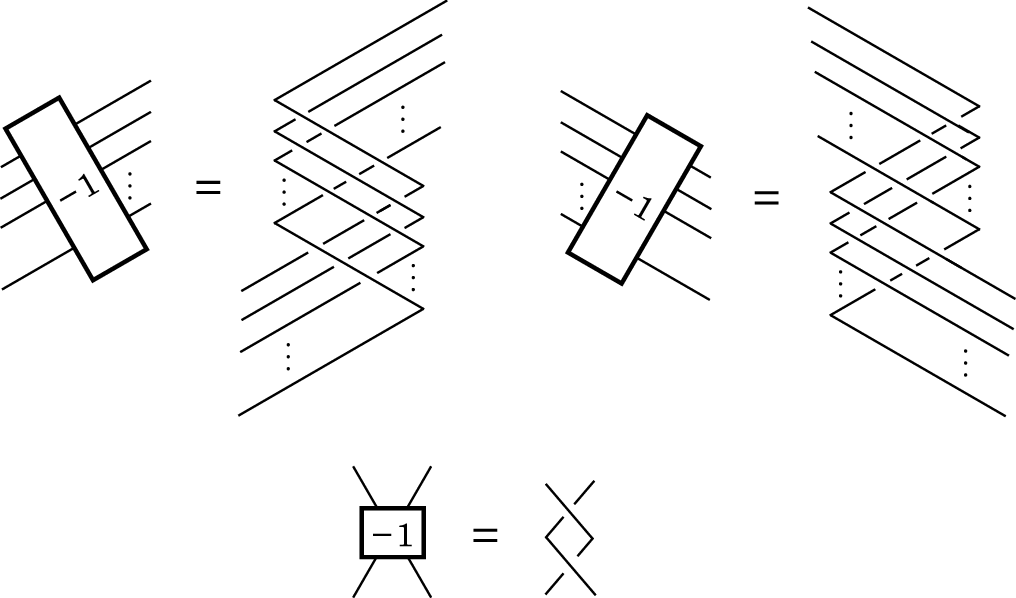}
\caption{Legendrian versions of a left-handed full twist}
\label{fig:Legendrian_left_twists}
\end{center}
\end{figure}

One can easily show the lemma below similarly to \cite{AM, AY1}. 
\begin{lemma}[\cite{AKMR}]\label{lem:V_n_cork}$(V_n, g_n)$ is a cork for any integer $n\geq 0$. 
\end{lemma}

We next introduce a new description of cork twists. For an integer $n$, let $V_n^*$ be the left contractible 4-manifold in Figure~\ref{fig:V_n_star}. Note that this manifold is diffeomorphic to the right 4-manifold by isotopy. We present a knot $K$ in $S^3$ by using a tangle $T_K$ as shown in Figure~\ref{fig:tangle_smooth}. We will show that a cork twist along $(V_n,g_n)$ can be obtained as in the lower side of Figure~\ref{fig:new_description}. Note that the upper side describes a cork twist along $(V_n,g_n)$. In particular, for any knot $K$ and any integers $n,k$, hooking the $k$-framed knot $K$ to the 1-handles of $V_n$ and $V_n^*$ in the same way has an effect of a cork twist along the obvious $(V_n,g_n)$. This is of independent interest because all known cork twists are described by an exchange of a dot and zero in the language of handlebody diagram (e.g.\ the upper side). Due to this description, one might expect that $V_n^*$ is an exotic copy of $V_n$. However, it turns out that $V_n^*$ is diffeomorphic to $V_n$. 

\begin{figure}[h!]
\begin{center}
\includegraphics[width=4.2in]{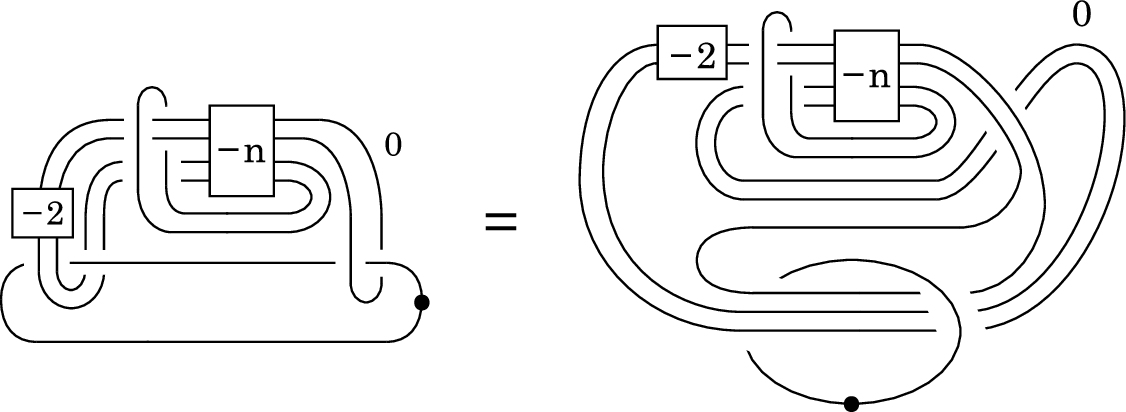}
\caption{Two diagrams of $V_n^*$}
\label{fig:V_n_star}
\end{center}
\end{figure}

\begin{figure}[h!]
\begin{center}
\includegraphics[width=0.8in]{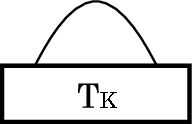}
\caption{A knot $K$ in $S^3$ given by a tangle $T_K$.}
\label{fig:tangle_smooth}
\end{center}
\end{figure}

\begin{figure}[h!]
\begin{center}
\includegraphics[width=3.9in]{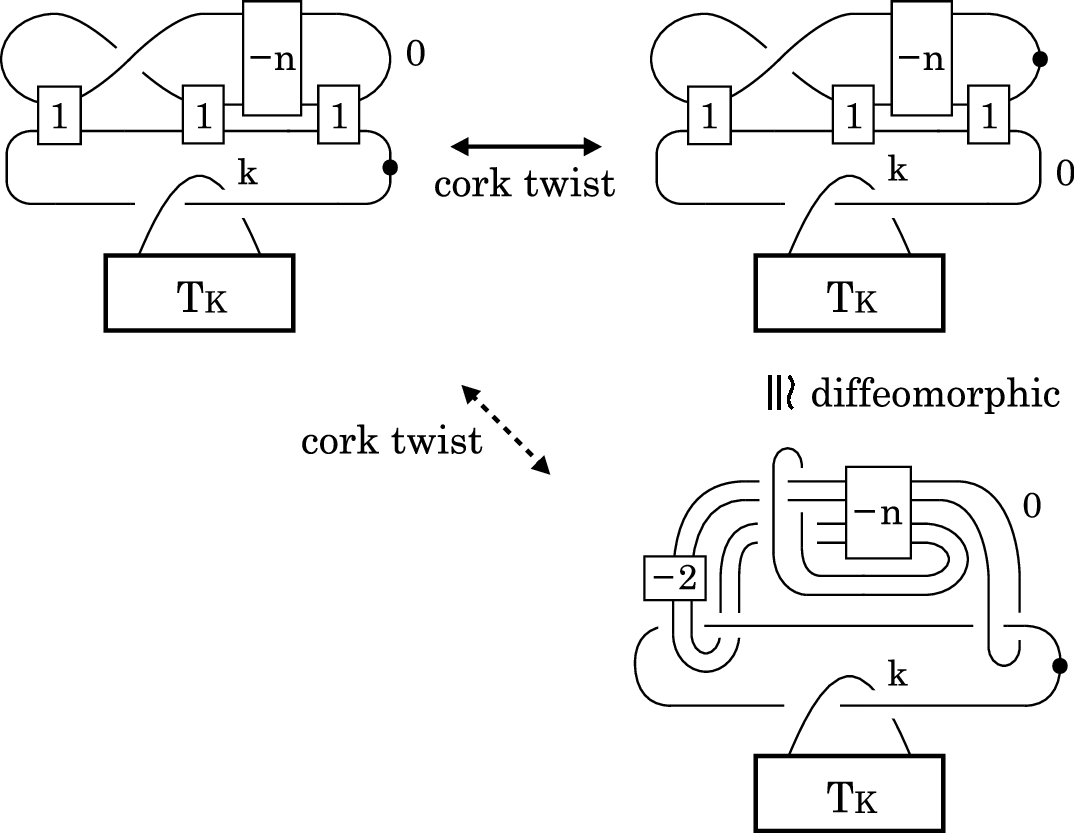}
\caption{A new description of a cork twist (lower side)}
\label{fig:new_description}
\end{center}
\end{figure}

To obtain the new description, we construct a diffeomorphism $\partial V_n\to \partial V_n^*$. For a knot $K$ in $S^3$, let $\gamma_{K}$ and $\gamma_{K}^*$ be knots in $\partial V_n$ and $\partial V_n^*$ given by Figure~\ref{fig:diffeo_star}, respectively. 

\begin{figure}[h!]
\begin{center}
\includegraphics[width=3.6in]{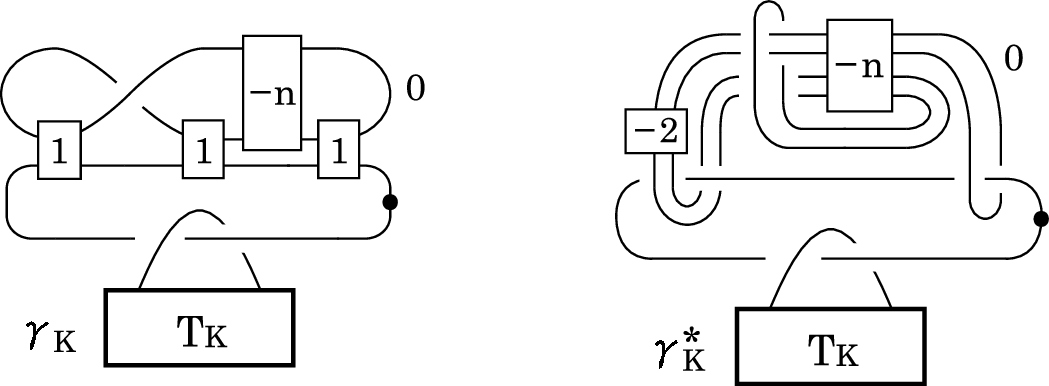}
\caption{Knots $\gamma_K$ and $\gamma_K^*$ in $\partial V_n$ and $\partial V_n^*$}
\label{fig:diffeo_star}
\end{center}
\end{figure}

\begin{theorem}\label{thm:boundary_V_n^*}For each integer $n$, there exists a diffeomorphism $g_n^*:\partial V_n\to \partial V_n^*$ satisfying the following conditions. 
\begin{itemize}
 \item The diffeomorphism $g_n^*$ sends the knot $\gamma_K$ to $\gamma_K^*$ for any knot $K$ in $S^3$. 
 \item The diffeomorphism $g_n^*\circ g_n^{-1}: \partial V_n\to \partial V_n^*$ extends to a diffeomorphism $V_n\to V_n^*$. In particular, $V_n^*$ is diffeomorphic to $V_n$. 
\end{itemize}
\end{theorem}
\begin{proof}Let $g_n^*:\partial V_n\to \partial V_n^*$ be the diffeomorphism defined by Figure~\ref{fig:new_description_proof}. (For the step from the seventh diagram to the eighth, use the isotopy in Figure~\ref{fig:isotopy_full_twist} after canceling the 1-handle.)  The first condition of the claim is obvious from the figure. 
The diffeomorphism $g_n^*\circ g_n^{-1}$ is clearly given by the procedure from the second diagram to the last diagram in the figure. Since each step is induced from a diffeomorphism between the obvious 4-manifolds, the second condition follows. 
\end{proof}

\begin{figure}[h!]
\begin{center}
\includegraphics[width=4.1in]{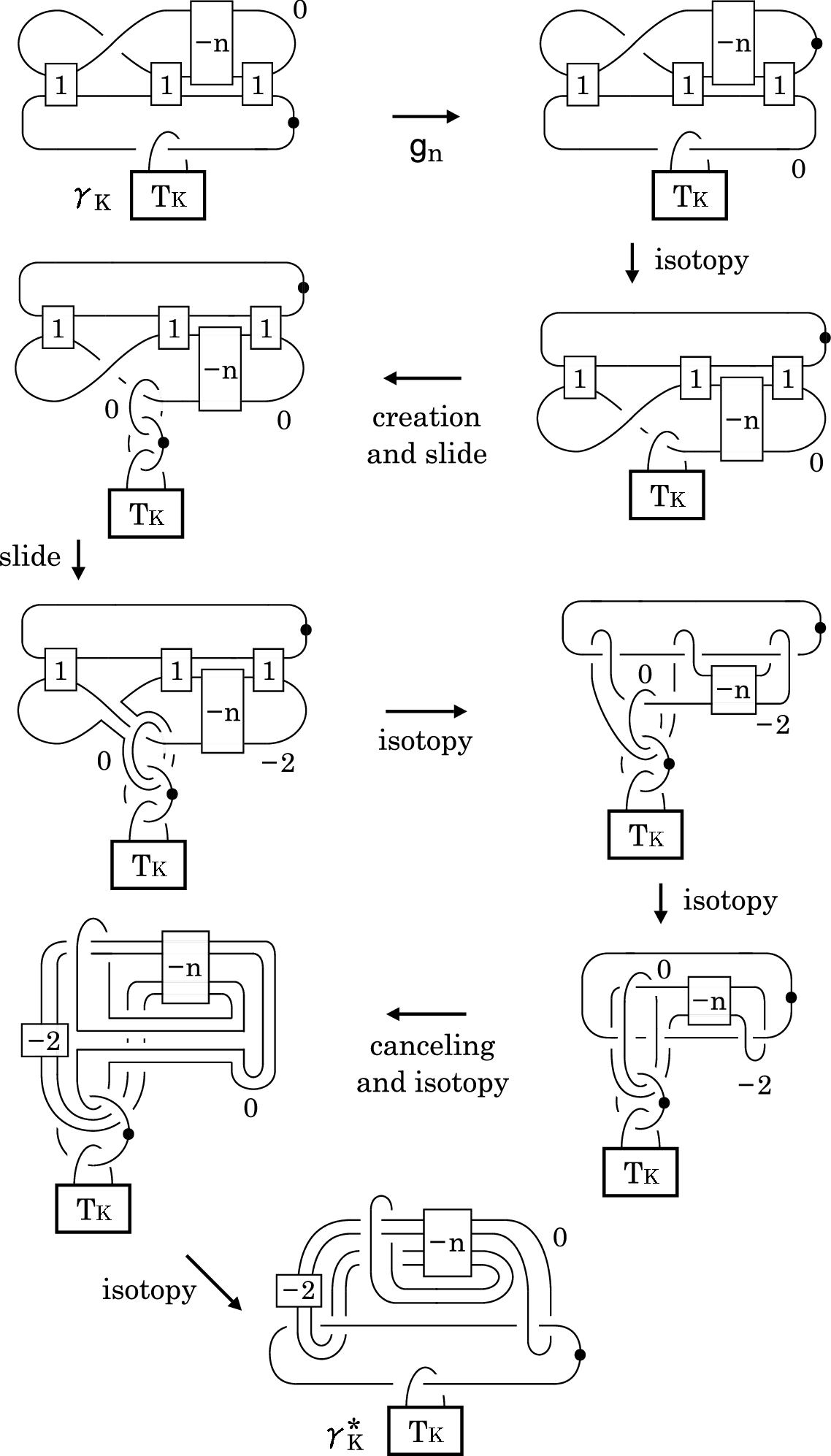}
\caption{The diffeomorphism $g_n^*:\partial V_n\to \partial V_n^*$}
\label{fig:new_description_proof}
\end{center}
\end{figure}

\begin{figure}[h!]
\begin{center}
\includegraphics[width=3.0in]{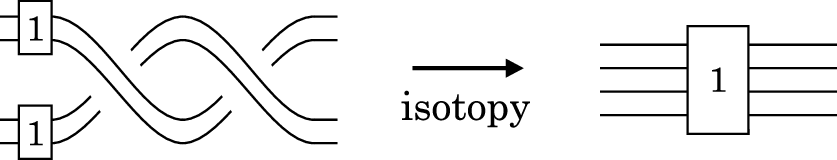}
\caption{Isotopy}
\label{fig:isotopy_full_twist}
\end{center}
\end{figure}
\begin{corollary}\label{cor:new_description}Assume that a smooth 4-manifold $X$ contains $V_n$ $(n\geq 0)$ as a submanifold. Then the 4-manifold obtained from $X$ by removing $V_n$ and gluing $V_n^*$ via the gluing map $g_n^*$ is diffeomorphic to the 4-manifold obtained from $X$ by the cork twist along $(V_n,g_n)$. 
\end{corollary}
\begin{proof}Since $g_n^*\circ g_n^{-1}: \partial V_n\to \partial V_n^*$ extends to a diffeomorphism $V_n\to V_n^*$, and $g_n^*=(g_n^*\circ g_n^{-1})\circ g_n$, the claim follows. 
\end{proof}
In summary, $V_n^*$ is diffeomorphic to $V_n$, and a cork twist along $(V_n,g_n)$ has the same effect as a surgery along $V_n$ via the gluing map $g_n^*:\partial V_n\to \partial V_n^*$, as shown in Figure~\ref{fig:new_description}.  
%%%%%%%%%%%%%%%%%%%%%
%%%%%%%%%%%%%%%%%%%
\section{Proofs of the main results}
In this section we prove our main results. For integers $n,m$, let $P_{n,m}$ and $Q_{n,m}$ be the knots in unknotted solid tori in $S^3$ given by Figure~\ref{fig:pattern_P_nm}. Here the dotted lines indicate solid tori, and the box $n$ denotes $|n|$ right-handed (resp.\ left-handed) full twists if $n$ is positive (resp.\ negative). In the case where $m=0$, they have the simple diagrams shown in Figure~\ref{fig:pattern}. 
We regard $P_{n,m}$ and $Q_{n,m}$ as (untwisted) satellite maps. Note that the knot $P_{n,m}(K)$ (resp.\ $Q_{n,m}(K)$) is isotopic to the $n$-twisted satellite of a knot $K$ with the pattern $P_{0,m}$ (resp.\ $Q_{0,m}$). Let $P_{n,m}^{(n)}(K)$ (resp.\ $Q_{n,m}^{(n)}(K)$) denote the 4-manifold represented by the knot $P_{n,m}(K)$ (resp.\ $Q_{n,m}(K)$) with $n$-framing. 

\begin{figure}[h!]
\begin{center}
\includegraphics[width=2.8in]{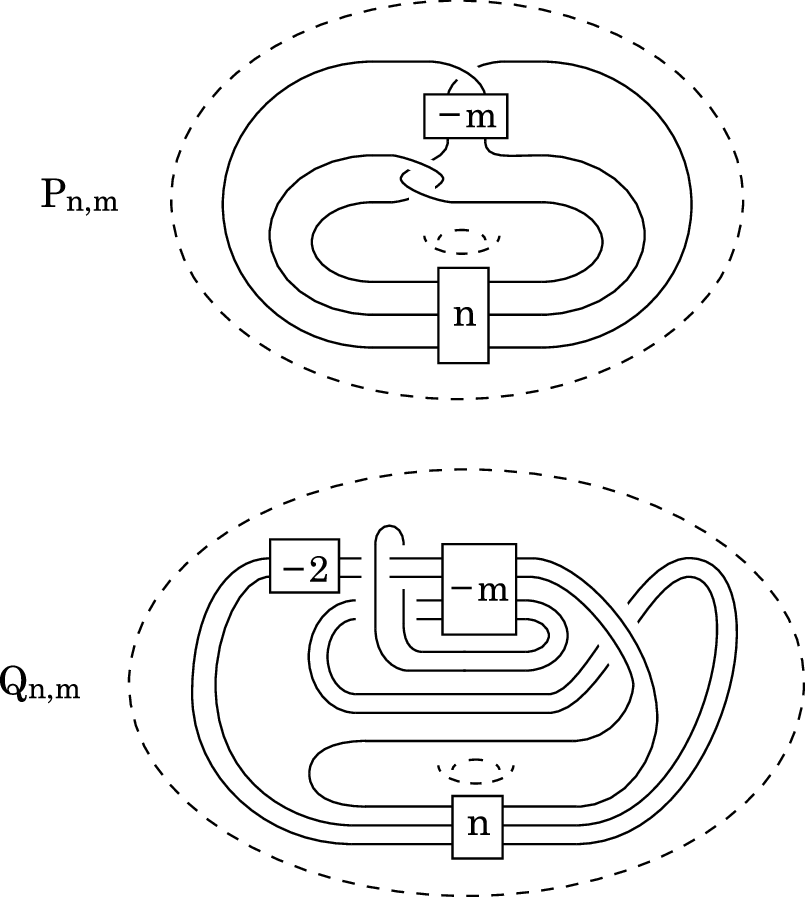}
\caption{Pattern knots $P_{n,m}$ and $Q_{n,m}$ $(n,m\in \mathbb{Z})$ in unknotted solid tori in $S^3$.}
\label{fig:pattern_P_nm}
\end{center}
\end{figure}
\begin{figure}[h!]
\begin{center}
\includegraphics[width=4.3in]{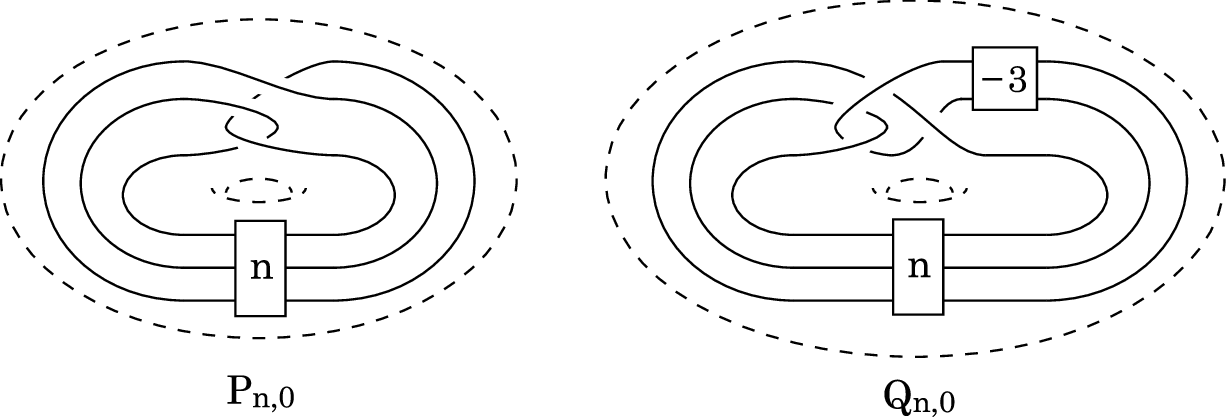}
\caption{Pattern knots $P_{n,0}$ and $Q_{n,0}$ $(n\in \mathbb{Z})$ in unknotted solid tori in $S^3$.}
\label{fig:pattern}
\end{center}
\end{figure}

Here we prove that the pair of satellite maps $P_{n,m}$ and $Q_{n,m}$ produces many pairs of exotic $n$-framed non-concordant knots. 
\begin{theorem}\label{thm:satellite:detail}Fix integers $n$ and $m$ with $m\geq 0$. For each knot $K$ in $S^3$ satisfying $2g_4(K)=\overline{ad}(K)+2$ and $n\leq \widehat{tb}(K)$, the knots $P_{n,m}(K)$ and $Q_{n,m}(K)$ satisfy the following conditions. 
\begin{itemize}
  \item The 4-manifolds $P_{n,m}^{(n)}(K)$ and $Q_{n,m}^{(n)}(K)$ are homeomorphic but not diffeomorphic to each other. Furthermore, both of these 4-manifolds admit Stein structures if $n\leq \overline{tb}(K)-1$. 
 \item $g_4(P_{n,m}(K))=g_4(K)+1$ and $g_4(Q_{n,m}(K))=g_4(K)$. 
\end{itemize}
Consequently, the knots $P_{n,m}(K)$ and $Q_{n,m}(K)$ are not concordant for any orientations, and their $n$-surgeries yield the same 3-manifold. 
\end{theorem}
\begin{remark}\label{thm:torus:remark}
$(1)$ There are infinitely many knots $K$ satisfying $2g_4(K)=\overline{ad}(K)+2$. For example, it is well-known that any positive $(p,q)$-torus knot $T_{p,q}$ satisfies this assumption, and $\widehat{tb}(T_{p,q})=\overline{tb}(T_{p,q})=pq-p-q$. This can be easily seen from an appropriate Legendrian realization and the adjunction inequality. Therefore, by Corollary in \cite{So}, this theorem produces infinitely many distinct pairs of exotic $n$-framed non-concordant knots for each integer $n$.
\smallskip\\
\noindent $(2)$ If a knot $K$ in $S^3$  satisfies the assumption of this theorem for fixed integers $n,m$, then both of the knots $P_{n,m}(K)$ and $Q_{n,m}(K)$ also satisfy the assumption (see Remark~\ref{rem:knots}). Therefore, just by iterating this operation to any knot satisfying the assumption, we obtain infinitely many distinct knots satisfying the assumption. For interesting applications of iterated satellite knots, see \cite{CR, Ray}. 
\smallskip\\
\noindent $(3)$ The knot $Q_{n,m}(K)$ is concordant to $K$ for any integers $n,m$. This can be easily seen from the diagram of $Q_{n,m}$ by checking that the pattern knot $Q_{n,m}$ is a band sum of the longitude of the solid torus and an unknot which are unlinking. 
\end{remark}

%\begin{remark}Shortly after this paper appeared on arXiv, Ray kindly informed the author that the value of $g_4(P_{n,m}(K))$ also follows from a result of Cochran-Ray~\cite{CR}, which was obtained by the argument similar to ours. (We note that our argument is based on \cite{AY5}.) Furthermore, they also discussed iterations of satellite operations. See also \cite{Ray}. \end{remark}

To prove this theorem, we show the lemma below. As seen from the proof, we obtain the satellite maps $P_{n,m}$ and $Q_{n,m}$ from the cork $V_m$ and its alternative description $V_m^*$. 
%%%%%%%%%

\begin{lemma}\label{proof:lem:cork}
For any integers $n, m$ and any knot $K$ in $S^3$, the 4-manifold $P_{n,m}^{(n)}(K)$ is homeomorphic to $Q_{n,m}^{(n)}(K)$. 
\end{lemma}
\begin{proof}
We present a knot $K$ using a tangle $T_K$ as in Figure~\ref{fig:tangle_smooth}. We can easily check that $P_{n,m}^{(n)}(K)$ and $Q_{n,m}^{(n)}(K)$ are respectively diffeomorphic to the left and the upper right 4-manifolds in Figure~\ref{fig:Q_nm_new_description}, by canceling the 1-handles (see also Figure~\ref{fig:Legendrian_V_n} and the right diagram in Figure~\ref{fig:V_n_star}). By Theorem~\ref{thm:boundary_V_n^*}, $Q_{n,m}^{(n)}(K)$ is obtained from $P_{n,m}^{(n)}(K)$ by removing $V_m$ and gluing $V_m^*$ via the gluing map $g_n^*$. Since any diffeomorphism between the boundaries of contractible smooth 4-manifolds extends to a homeomorphism between the 4-manifolds (\cite{Fr}), the claim follows. 
\end{proof}
\begin{figure}[h!]
\begin{center}
\includegraphics[width=4.1in]{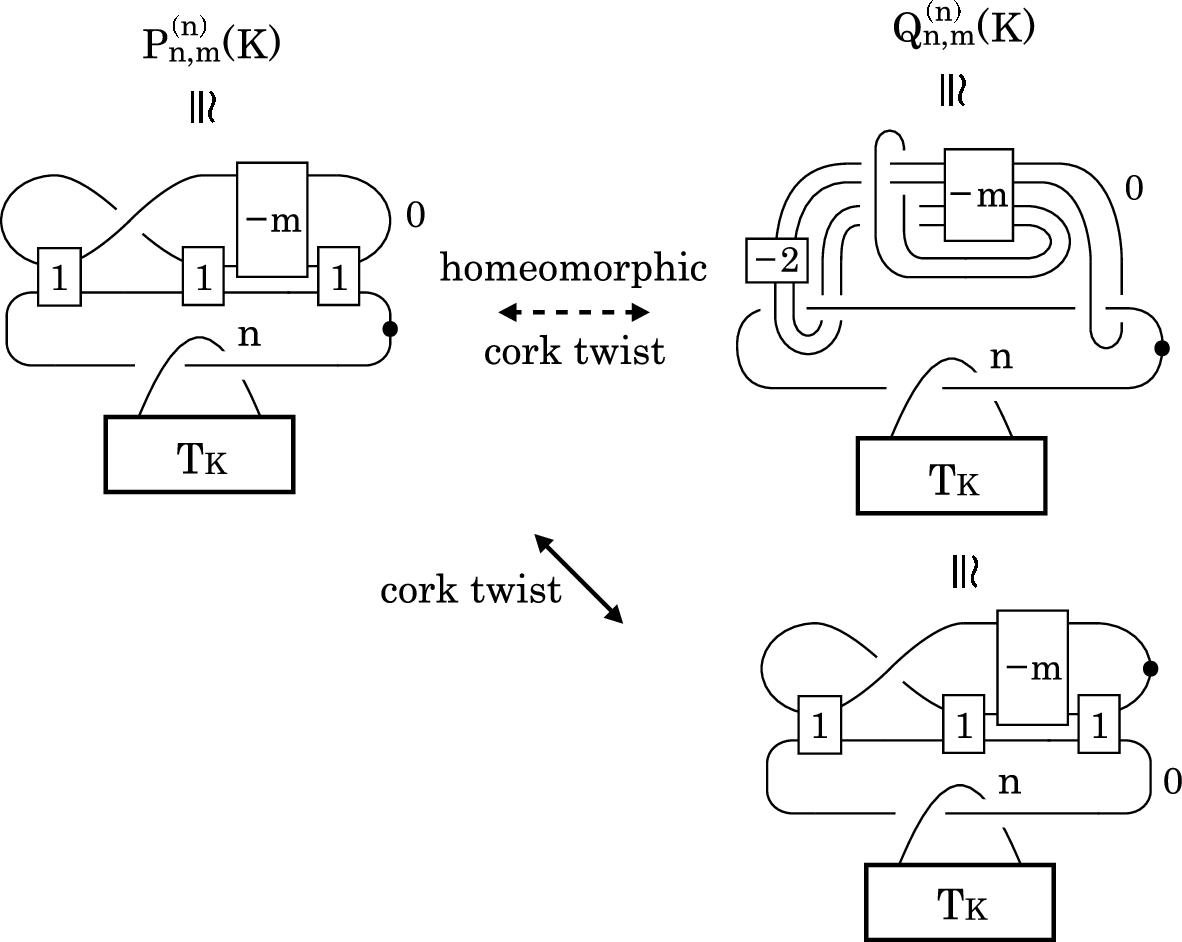}
\caption{The left and the right 4-manifolds are diffeomorphic to $P_{n,m}^{(n)}(K)$ and $Q_{n,m}^{(n)}(K)$, respectively.}
\label{fig:Q_nm_new_description}
\end{center}
\end{figure}

\begin{remark}It follows from Corollary~\ref{cor:new_description} (see also Figure~\ref{fig:Q_nm_new_description}) that $P_{n,m}^{(n)}(K)$ is obtained from $Q_{n,m}^{(n)}(K)$ by a cork twist along $(V_m, g_m)$ (in the case $m\geq 0$). 
\end{remark}

Akbulut and the author \cite{AY5} gave an algorithm which produces arbitrarily many exotic Stein 4-manifolds  by applying corks. Adapting the argument to our simple case, we prove Theorem~\ref{thm:satellite:detail}. 
\begin{proof}[Proof of Theorem~\ref{thm:satellite:detail}]
Fix integers $n$ and $m$ with $m\geq 0$. Let $K$ be a knot in $S^3$ satisfying $2g_4(K)=\overline{ad}(K)+2$ and $n\leq \widehat{tb}(K)$. We first give Legendrian representatives of the satellite knots $P_{n,m}(K)$ and $Q_{n,m}(K)$. 
Due to the assumption on $K$, there exists a Legendrian representative of $K$ with $tb=\widehat{tb}(K)$ and $ad=\overline{ad}(K)$. Since $n\leq \widehat{tb}(K)$, by adding zig-zags to a front diagram of the representative, we get a Legendrian representative $\mathcal{K}$ of $K$ satisfying $n=tb(\mathcal{K})$ and $ad(\mathcal{K})=\overline{ad}(K)$. We present a front diagram of $\mathcal{K}$ by a Legendrian tangle $\mathcal{T}_{\mathcal{K}}$ as in the left diagram of Figure~\ref{fig:tangle_legendrian_K}. We then draw the front diagram consisting of three copies of $\mathcal{K}$ each of which is slightly shifted to the vertical direction (cf.\ upper Legendrian pictures in Figure~\ref{fig:Legendrian_left_twists}). We present the resulting front diagram by a Legendrian tangle $\mathcal{T}_{\mathcal{K}}^3$ as in the right diagram. Using this tangle, we obtain the Legendrian representatives of $P_{n,m}(K)$ and $Q_{n,m}(K)$ in Figures~\ref{fig:Legendrian_P_nm} and \ref{fig:Legendrian_Q_nm}, respectively. Here the left-handed full twists in the boxes denote the Legendrian versions shown in Figure~\ref{fig:Legendrian_left_twists}. %Unless otherwise stated, we regard $P_n(K)$ and $Q_n(K)$ as Legendrian knots by these representatives. 

\begin{figure}[h!]
\begin{center}
\includegraphics[width=2.8in]{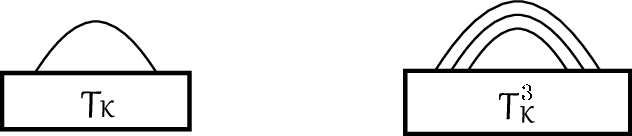}
\caption{The left is a front diagram of $\mathcal{K}$. The right is a front diagram consisting of three copies of $\mathcal{K}$ slightly shifted to the vertical direction.}
\label{fig:tangle_legendrian_K}
\end{center}
\end{figure}

\begin{figure}[h!]
\begin{center}
\includegraphics[width=3.0in]{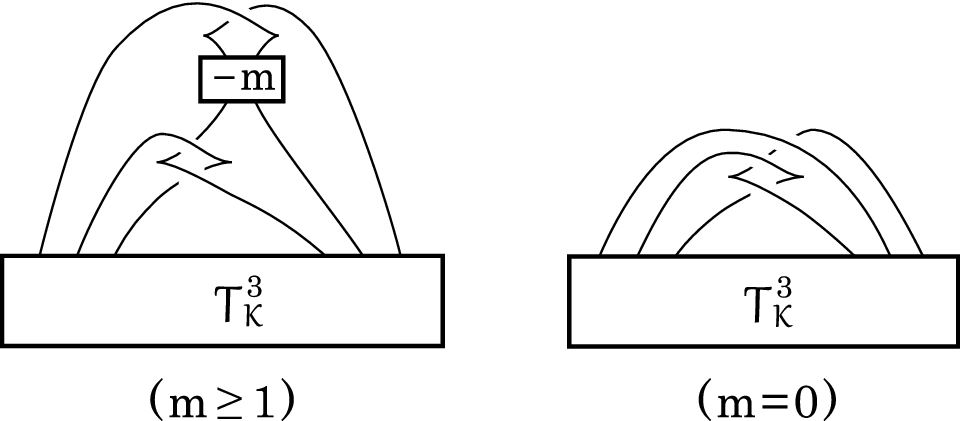}
\caption{A Legendrian representative of $P_{n,m}(K)$ $(m\geq 0)$}
\label{fig:Legendrian_P_nm}
\end{center}
\end{figure}

\begin{figure}[h!]
\begin{center}
\includegraphics[width=1.8in]{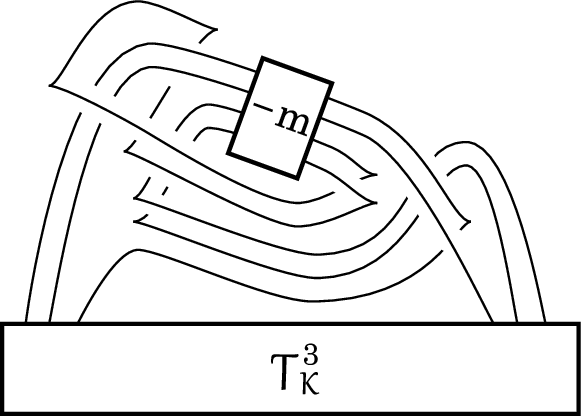}
\caption{A Legendrian representative of $Q_{n,m}(K)$ $(m\geq 0)$}
\label{fig:Legendrian_Q_nm}
\end{center}
\end{figure}

Next we determine the 4-genus and the $n$-shake genus of $P_{n,m}(K)$. By counting the writhe and the number of left cusps of the front diagram, one can easily check that the Legendrian representative of $P_{n,m}(K)$ satisfies $tb=tb(\mathcal{K})+2$ and $|r|=|r(\mathcal{K})|$. By adding a zig-zag to this front diagram, we get a Legendrian representative of $P_{n,m}(K)$ with $tb=n+1$ and $|r|=|r(\mathcal{K})|+1$. Applying the adjunction inequality, we see 
\begin{equation*}
n \,+\, |r(\mathcal{K})|+1\leq 2g_s^{(n)}(P_{n,m}(K))-2. 
\end{equation*}
Since $n-1+|r(\mathcal{K})|=\overline{ad}(K)$ and $2g_4(K)=\overline{ad}(K)+2$, this inequality implies $g_s^{(n)}(P_{n,m}(K))\geq g_4(K)+1$. On the other hand, we see that $P_{n,m}(K)$ becomes isotopic to $\mathcal{K}$ after changing the lowest crossing shown in Figure~\ref{fig:Legendrian_P_nm}. This implies that $P_{n,m}(K)$ bounds a surface of genus $g_4(\mathcal{K})+1$ in $D^4$. Therefore, we obtain 
\begin{equation*}
g_s^{(n)}(P_{n,m}(K))=g_4(P_{n,m}(K))= g_4(K)+1. 
\end{equation*}
Since $Q_{n,m}(K)$ is concordant to $K$ (Remark~\ref{thm:torus:remark}), we also see 
\begin{equation*}
g_4(Q_{n,m}(K))=g_4(K). 
\end{equation*}

Here we distinguish smooth structures on $P_{n,m}^{(n)}(K)$ and $Q_{n,m}^{(n)}(K)$. By Lemma~\ref{proof:lem:cork}, the 4-manifold $P_{n,m}^{(n)}(K)$ is homeomorphic to $Q_{n,m}^{(n)}(K)$. On the other hand, the above arguments show
\begin{equation*}
g_s^{(n)}(P_{n,m}(K))=g_4(K)+1>g_4(Q_{n,m}(K))\geq g_s^{(n)}(Q_{n,m}(K)).
\end{equation*}
Consequently, $g_s^{(n)}(P_{n,m}(K))\neq g_s^{(n)}(Q_{n,m}(K))$. Therefore, it follows from the definition of the $n$-shake genus that $P_{n,m}^{(n)}(K)$ is not diffeomorphic to $Q_{n,m}^{(n)}(K)$. 

Lastly we check existence of Stein structures on $P_{n,m}^{(n)}(K)$ and $Q_{n,m}^{(n)}(K)$. By using the left and the lower right diagrams of $P_{n,m}^{(n)}(K)$ and $Q_{n,m}^{(n)}(K)$ in Figure~\ref{fig:Q_nm_new_description} and the Stein handlebody diagram of $V_n$ in Figure~\ref{fig:Legendrian_V_n}, we can easily realize $P_{n,m}^{(n)}(K)$ (resp.\ $Q_{n,m}^{(n)}(K)$) as a Stein handlebody for $n\leq \overline{tb}(K)$ (resp.\ $n\leq \overline{tb}(K)-1$). Hence, according to \cite{E1, G1}, these 4-manifolds admit Stein structures. 
\end{proof}
%%%%%%%%%%%%%
%%%%%%%%%%%%%%%
%%%%%%%%%%%%

\begin{remark}[Knots $P_{n,m}(K)$ and $Q_{n,m}(K)$]\label{rem:knots}
Fix integers $n,m$ with $m\geq 0$. For a knot $K$ satisfying the assumption of Theorem~\ref{thm:satellite:detail}, the knots $P_{n,m}(K)$ and $Q_{n,m}(K)$ satisfy the conditions below, as seen from the proof. 
\begin{align*}
2g_4(P_{n,m}(K))&=\overline{ad}(P_{n,m}(K))+2,\quad \widehat{tb}(P_{n,m}(K))\geq n+2.\\
2g_4(Q_{n,m}(K))&=\overline{ad}(Q_{n,m}(K))+2, \quad \widehat{tb}(Q_{n,m}(K))\geq n. 
\end{align*}
\end{remark}
%%%%%%%%%%%
%%%%%%%%%%%

%We can easily determine the $n$-shake genera of $P_{n,m}(K)$ and $Q_{n,m}(K)$. Furthermore, by an argument similar to \cite{CFHH}, we can calculate the values of the concordance invariants $\tau$ and $s$. 

\begin{remark}[The $n$-shake genera]\label{cor:tau;shake1}Fix integers $n,m$ with $m\geq 0$. Assume that a knot $K$ satisfies the assumption of Theorem~\ref{thm:satellite:detail}. Then the $n$-shake genera of $P_{n,m}(K)$ and $Q_{n,m}(K)$ are given as follows. 
\begin{align*}g_s^{(n)}(P_{n,m}(K))=g_4(K)+1,\quad  \text{if  $n\leq \widehat{tb}(K)$.}\\
g_s^{(n)}(Q_{n,m}(K))=g_4(K),\quad  \text{if $n\leq \widehat{tb}(K)-1$.}
\end{align*}
The former equality is obvious from the proof of Theorem~\ref{thm:satellite:detail}. The latter can be seen as follows. It is easy to realize the lower right handlebody in Figure~\ref{fig:Q_nm_new_description} as a Stein handlebody, if $n\leq \widehat{tb}(K)-1$. Therefore, this manifold admits a Stein structure. Applying the adjunction inequality for Stein 4-manifolds (\cite{AM, LM2}. cf.\ \cite{GS, OS1}) to this 4-manifold, we obtain the above equality. 
\end{remark}
\begin{remark}[The knot concordance invariants $\tau$ and $s$]\label{cor:tau;shake2}Let $n,m,K$ be as in Remark~\ref{cor:tau;shake1}. Then $\tau$ and $s$ of $P_{n,m}(K)$ and $Q_{n,m}(K)$ are given as follows. 
\begin{gather*}
2\tau(P_{n,m}(K))=s(P_{n,m}(K))=2g_4(P_{n,m}(K))=\overline{ad}(K)+4, \\
2\tau(Q_{n,m}(K))=s(Q_{n,m}(K))=2g_4(Q_{n,m}(K))=\overline{ad}(K)+2.
\end{gather*}
These are straightforward from Remark~\ref{rem:knots} and the following inequalities for an arbitrary knot $L$ in $S^3$ (\cite{Pl, Pl_Kho, Shu}. cf.\ \cite{CFHH}). 
\begin{equation*}
 \overline{ad}(L)\leq 2\tau(L)-2\leq 2g_4(L)-2,\qquad 
 \overline{ad}(L)\leq s(L)-2\leq 2g_4(L)-2.
\end{equation*}
\end{remark}
\begin{remark}$(1)$ The $n$-shake genus (and thus 4-genus) of $P_{n,m}(K)$ can be obtained also by applying the original argument of Akbulut and the author~\cite{AY5}, that is, by applying the adjunction inequality to the (Stein version of) left handlebody in Figure~\ref{fig:Q_nm_new_description}. This method clearly works for many other satellite knots. The argument of this paper is a simplification of this method by canceling the 1-handle.\smallskip\\
$(2)$ As pointed out by Ray, the above values of $\widehat{tb}$, the $n$-shake genus, the 4-genus, $\tau$ and $s$ for $P_{n,m}(K)$ follow from results of Cochran-Ray \cite{CR}, which were obtained by similar arguments. 
\end{remark}

%\begin{corollary}For each odd integer $n\geq -1$, there exists a pair of $n$-framed knots representing homeomorphic 4-manifolds such that one 4-manifold admits a Stein structure, but the other does not admit any Stein structure.  \end{corollary}
%%%%%%%%%%
%%%%%%%%%%%
We give a simple example in the case where the framing is $0$. 
\begin{example}We consider the right-handed trefoil knot $T_{2,3}$. This knot satisfies the assumption of Theorem~\ref{thm:satellite:detail} and $\widehat{tb}(T_{2,3})=\overline{tb}(T_{2,3})=1$. Hence by Theorem~\ref{thm:satellite:detail}, the knots $P_{0,0}(T_{2,3})$ and $Q_{0,0}(T_{2,3})$ are non-concordant for any orientations, and their $0$-surgeries yield the same 3-manifold. Furthermore, these knots with $0$-framings represent Stein 4-manifolds which are homeomorphic but non-diffeomorphic to each other. Figure~\ref{fig:example} gives diagrams of these knots. 
\end{example}
\begin{figure}[h!]
\begin{center}
\includegraphics[width=3.8in]{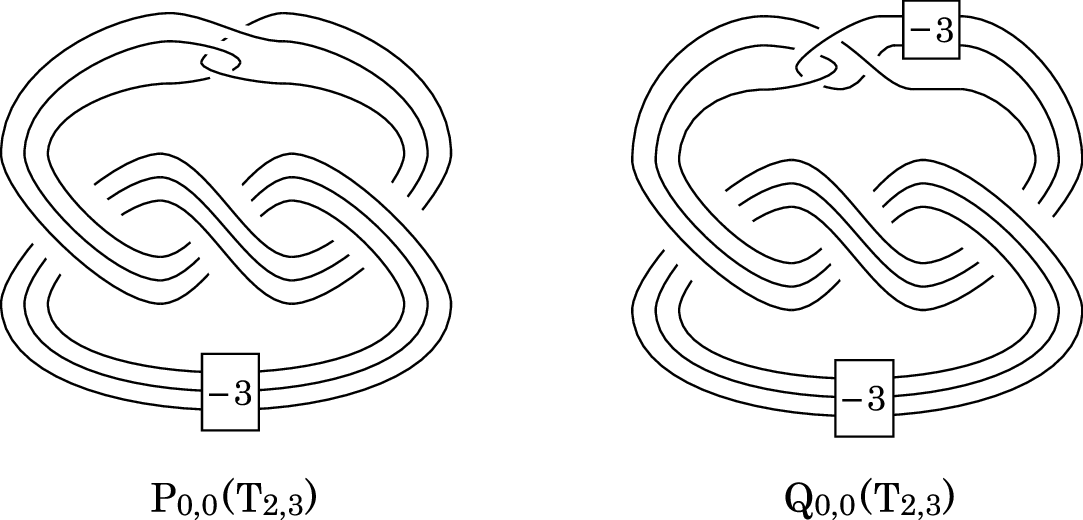}
\caption{$P_{0,0}(T_{2,3})$ and $Q_{0,0}(T_{2,3})$}
\label{fig:example}
\end{center}
\end{figure}

Now we can easily prove our main results. 

\begin{proof}[Proof of Theorems~\ref{intro:thm:knot}, \ref{intro:thm:satellite_knot} and \ref{thm:intro}]These are obvious from Theorem~\ref{thm:satellite:detail} and Remark~\ref{thm:torus:remark}. Note that we can distinguish the pairs of knots by comparing their 4-genera. 
\end{proof}

\begin{proof}[Proof of Corollary~\ref{intro:cor:4-dim satellite}]This is obvious from Lemma~\ref{proof:lem:cork}, Theorem~\ref{thm:satellite:detail} and Remark~\ref{thm:torus:remark}.
\end{proof}

%\begin{proof}[Proof of Theorem~] Let $p,q$ be relatively prime integers with $p\geq 2$ and $q\geq 2$. By Theorem~\ref{thm:satellite:detail} and Remark~\ref{thm:torus:remark}, we obtain two knots with the same $0$-surgery such that their 4-genera are $\frac{(p-1)(q-1)}{2}$ and $\frac{(p-1)(q-1)}{2}+1$. Therefore, we obtain infinitely many pairs of knots satisfying the condition in Theorem~\ref{thm:intro} by varying $p,q$. \end{proof}

Regarding Stein structures, Theorem~\ref{thm:satellite:detail} also gives different types of exotic framed knots. 
\begin{corollary}\label{cor:Stein_non-Stein}For each odd integer $n\geq -1$, there exists a pair of $n$-framed knots representing a pair of homeomorphic but non-diffeomorphic 4-manifolds such that one 4-manifold admits a Stein structure, but the other 4-manifold admits no Stein structure. Furthermore, there exist infinitely many distinct pairs of such $n$-framed knots. 
\end{corollary}
\begin{proof} We use the following proposition. 
\begin{proposition}\label{prop:cor:Stein_non-Stein}
Fix an integer $m\geq 0$. For an odd integer $n\geq -1$, assume that a knot $K$ satisfies $n=\overline{tb}(K)=2g_4(K)-1$ $($e.g.\ the torus knot $T_{2,2n+1}$$)$. Then the 4-manifolds $P_{n,m}^{(n)}(K)$ and $Q_{n,m}^{(n)}(K)$ are homeomorphic but not diffeomorphic to each other. Furthermore, $P_{n,m}^{(n)}(K)$ admits a Stein structure, but $Q_{n,m}^{(n)}(K)$ does not admit any Stein structure. 
\end{proposition}
\begin{proof}[Proof of Proposition~\ref{prop:cor:Stein_non-Stein}]
The former claim follows from Theorem~\ref{thm:satellite:detail}. The proof of Theorem~\ref{thm:satellite:detail} shows that $P_{n,m}^{(n)}(K)$ admits a Stein structure. By the inequality 
\begin{equation*}
g_s^{(n)}(Q_{n,m}(K))\leq g_4(Q_{n,m}(K))=g_4(K)
\end{equation*}
and the assumption on $K$, we see 
\begin{equation*}
2g_s^{(n)}(Q_{n,m}(K))-2<n. 
\end{equation*}
Since $n$ is the self-intersection number of the generator of $H_2(Q_{n,m}^{(n)}(K);\mathbb{Z})$, the adjunction inequality for Stein 4-manifolds (\cite{AM, LM2}. cf.\ \cite{GS, OS1}) guarantees that $Q_{n,m}^{(n)}(K)$ does not admit any Stein structure. 
\end{proof}
Now let $n,m$ be as in the above proposition, and let $K_1, K_2, \dots $ be infinitely many distinct knots in $S^3$ satisfying $n=\overline{tb}(K_i)=2g_4(K_i)-1$ for each $i$. Such an infinite family can be constructed, for example, as follows. Let $J$ be a non-trivial ribbon knot with $\overline{tb}(J)=-1$ (e.g.\ the knot $K_m$ $(m\leq -1)$ in \cite{Y16}), and let $K_i$ be the connected sum of the torus knot $T_{2, 2n+1}$ and $i-1$ copies of $J$. Then it is easy to see $n=\overline{tb}(K_i)=2g_4(K_i)-1$, and the unique prime factorization theorem of knots tells that $K_1, K_2, \dots $ are pairwise distinct. By Corollary in \cite{So}, we can easily show that infinitely many knots $P_{n,m}(K_1), P_{n,m}(K_2), \dots $ are pairwise distinct. The above proposition tells that each pair of $n$-framed knots $P_{n,m}(K_i)$ and $Q_{n,m}(K_i)$ satisfies the condition of the desired corollary, and thus we obtain infinitely many distinct such pairs by varying $i$. This completes the proof of Corollary~\ref{cor:Stein_non-Stein}. 
\end{proof}

\begin{corollary}\label{both_non-Stein}For each integer $n$, there exist infinitely many distinct pairs of $n$-framed knots such that each pair represents a pair of homeomorphic but non-diffeomorphic 4-manifolds admitting no Stein structures. 
\end{corollary}
\begin{proof}Let $n, m$ be fixed integers with $m\geq 0$. Assume that a knot $K$ satisfies $2g_4(K)=\overline{ad}(K)+2$, $n\leq \widehat{tb}(K)$, and $n\leq \overline{tb}(K)-1$. We denote the mirror images of the knots $P_{n,m}(K)$ and $Q_{n,m}(K)$ by $\overline{P_{n,m}(K)}$ and $\overline{Q_{n,m}(K)}$. We here show that the 4-manifold represented by $\overline{P_{n,m}(K)}$ with $(-n)$-framing does not admit any Stein structure. Suppose, to the contrary, that this 4-manifold $X$ admits a Stein structure. Then the boundary connected sum $Z:=X\natural P_{n,m}^{(n)}(K)$ admits a Stein structure, since $P_{n,m}^{(n)}(K)$ admits a Stein structure. By a handle slide, we see that $Z$ contains an embedded 2-sphere representing a non-zero second homology class with the self-intersection number $0$. (Note that the connected sum $P_{n,m}(K)\#\overline{P_{n,m}(K)}$ is a ribbon knot.) This contradicts the adjunction inequality for Stein 4-manifolds (\cite{AM, LM2}. cf.\ \cite{GS, OS1}). Hence $X$ does not admit any Stein structure. The same argument clearly works for $\overline{Q_{n,m}(K)}$. Therefore, the knots $\overline{P_{n,m}(K)}$ and $\overline{Q_{n,m}(K)}$ with $(-n)$-framings represent non-Stein 4-manifolds. Since these non-Stein 4-manifolds are the 4-manifolds $P_{n,m}^{(n)}(K)$ and $Q_{n,m}^{(n)}(K)$ with the reverse orientations, Theorem~\ref{thm:satellite:detail} and Remark~\ref{thm:torus:remark} imply the claim by varying $K$.  
\end{proof}

One might ask whether our counterexamples to the Akbulut-Kirby conjecture are topologically concordant. We point out that there are infinitely many topologically concordant examples among them by an argument similar to \cite{CFHH, Ray}. 

\begin{corollary}There exists a pair of knots with the same $0$-surgery which are topologically concordant but smoothly non-concordant for any orientations. Furthermore, there exist infinitely many distinct pairs of such knots. 
\end{corollary}
\begin{proof}Let $m\geq 0$ be a fixed integer, and let $K$ be a topologically slice knot with $2g_4(K)=\overline{ad}(K)+2$ and $0\leq \widehat{tb}(K)$. For example, the untwisted Whitehead double of a positive torus knot satisfies this condition (e.g.\ \cite{AM}). Note that the Whitehead double of a knot is topologically slice as is well-known. It is known that if two knots are topologically concordant, then the images of them by a satellite map are also topologically concordant (e.g.\ \cite{CDR}). Hence $P_{0,m}(K)$ and $Q_{0,m}(K)$ are topologically slice knots. Therefore, by Theorem~\ref{thm:satellite:detail}, the pair of knots $P_{0,m}(K)$ and $Q_{0,m}(K)$ satisfies the claim. Remark~\ref{thm:torus:remark} thus tells that we obtain infinitely many distinct such pairs by iterating this construction. 
\end{proof}

Let us recall that each of our counterexamples to the Akbulut-Kirby conjecture is a pair of knots $P_{0,m}(K)$ and $Q_{0,m}(K)$ $(m\geq 0)$, where $K$ is an arbitrary knot satisfying $2g_4(K)=\overline{ad}(K)+2$ and $0\leq \widehat{tb}(K)$. Levine pointed out to the author that, at least in the $m=0$ case, the condition on $K$ can be relaxed by combining a result of his paper~\cite{Lev}, namely, the corollary below holds. %Here $\epsilon$ is an invariant of knot concordance defined by Hom~\cite{Hom}. 

\begin{corollary}\label{cor:epsilon}Assume that a knot $K$ satisfies either $\tau(K)>0$ or $\epsilon(K)=-1$. Then the knots $P_{0,0}(K)$ and $Q_{0,0}(K)$ have the same $0$-surgery but are not concordant for any orientations. 
\end{corollary} 
\begin{proof}A result of Levine \cite{Lev} shows $\tau(P_{0,0}(K))=\tau(K)+1$ (and $\epsilon(P_{0,0}(K))=1$) under the above assumption. Hence $P_{0,0}(K)$ is not concordant to $K$ for any orientations. Since $Q_{0,0}(K)$ is concordant to $K$ (Remark~\ref{thm:torus:remark}), this fact and Lemma~\ref{proof:lem:cork} shows the claim. 
\end{proof}
We remark that this corollary enlarges the class of $K$ from the class given by Theorem~\ref{thm:satellite:detail}. For example, according to \cite{Hom}, a negative torus knot $T_{p,-q}$ $(p,q>1)$ satisfies $\epsilon=-1$ and hence the assumption of this corollary, though it does not satisfy the aforementioned condition given by Theorem~\ref{thm:satellite:detail}. 

\begin{proof}[Proof of Corollary~\ref{intro:cor:tau}]This is obvious from Theorem~\ref{thm:satellite:detail}, Remarks~\ref{thm:torus:remark} and \ref{cor:tau;shake2}, and the proof of Corollary~\ref{cor:epsilon}.
\end{proof}

\begin{remark}[Extension of the main construction]
For simplicity, we defined the cork $(V_m,g_m)$ and the pattern knots $P_{n,m}$ and $Q_{n,m}$ in the case where $m$ is an integer. Clearly, we can extend these definitions to the case where $m$ is a half-integer, and we can similarly prove our results for any positive half-integer $m$. We remark that $(V_{\frac{1}{2}}, g_{\frac{1}{2}})$ is the same cork as $(\overline{W}_1, \overline{f}_1)$ in \cite{AY1}. As pointed out by a referee, our cork $(V_m,g_m)$ coincides with the cork $(\overline{C}_h,\tau)$ $(h=-m)$ earlier introduced by Auckly-Kim-Melvin-Ruberman \cite{AKMR}. We remark that many more examples can be constructed by using Proposition~\ref{prop:hook}, which is introduced in the next section. 
\end{remark}

\section{Dot-zero surgery and hook surgery}\label{sec:hook_dot}
In this section, we first discuss a certain surgery which we call a dot-zero surgery. This surgery is a generalization of cork twists along Mazur type corks. Specifically we give a sufficient condition on a link in $S^3$ such that any dot-zero surgery induced from the link does not change the smooth structure of a 4-manifold. Applying this result, we prove Theorems~\ref{intro:thm:noncork} and \ref{intro:thm:cork_noncork}. We next introduce a hook surgery as a formulation of the new description of cork twists obtained in Section~\ref{sec:new_description}. We then give a sufficient condition that a dot-zero surgery (including Mazur type cork twists) admits a hook surgery description.

%%%%%%%%%%%%%%%%%%%%
%%%%%%%%%%%%%%%%%%%
\subsection{Dot-zero surgery}
For an ordered link $L$ of two unknotted components in $S^3$, we denote the link $L$ with the reverse order by $\widetilde{L}$. Let $X_L$ be the 4-dimensional handlebody obtained from $L$ by putting a dot on the first component and a $0$ on top of the second component. Figure~\ref{fig:hook} is a diagram of this handlebody, where the link $L$ is described using a tangle $T_L$. 
\begin{figure}[h!]
\begin{center}
\includegraphics[width=1.3in]{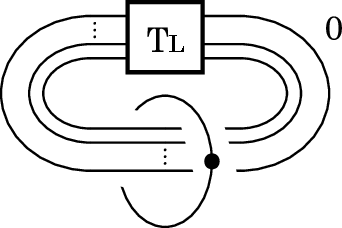}
\caption{$X_L$}
\label{fig:hook}
\end{center}
\end{figure}

By exchanging the dot and zero, we obtain the 4-manifold $X_{\widetilde{L}}$. Since $X_{\widetilde{L}}$ is obtained from $X_L$ by surgering $S^1\times D^3$ to $D^2\times S^2$ and then surgering the other $D^2\times S^2$ to $S^1\times D^3$, this operation induces a diffeomorphism $\varphi_L:\partial X_L\to \partial X_{\widetilde{L}}$. 
For a smooth 4-manifold $Z$ containing $X_L$ as a submanifold, remove $X_L$ and glue $X_{\widetilde{L}}$ back via the gluing map $\varphi_L$. We say that the resulting 4-manifold is obtained from $Z$ by a \textit{dot-zero surgery} along $(X_L,\varphi_L)$. 

It is known that many important surgeries are essentially dot-zero surgeries along (often Stein) 4-manifolds. Logarithmic transform, Fintushel-Stern's rational blowdown, cork twists and plug twists are such examples (cf.\ \cite{A_book, AY1, GS}). Therefore, characterizing a link which can (or cannot) alter diffeomorphism types preserving homeomorphism types is a natural problem. The characterization is also helpful to find a diffeomorphism between complicated handlebodies.

Here we give a sufficient condition on a link such that any dot-zero surgery induced from the link does not change the diffeomorphism type of a 4-manifold. 

\begin{proposition}\label{prop:dot-zero}Suppose that an ordered link $L$ of two unknotted components in $S^3$ becomes a trivial link after a single crossing change between two components in a diagram of $L$. Then the diffeomorphism $\varphi_L: \partial X_L\to \partial X_{\widetilde{L}}$ extends to a diffeomorphism $X_L\to X_{\widetilde{L}}$. Consequently, any dot-zero surgery along $(X_L,\varphi_L)$ does not change the diffeomorphism type of a 4-manifold. 
\end{proposition}
\begin{proof}We present $L$ as in the left diagram of Figure~\ref{fig:associated_link} using a tangle $T_L$. By the assumption, we may assume that the associated link $L'$ given by the right diagram is a trivial link. 

\begin{figure}[h!]
\begin{center}
\includegraphics[width=3.2in]{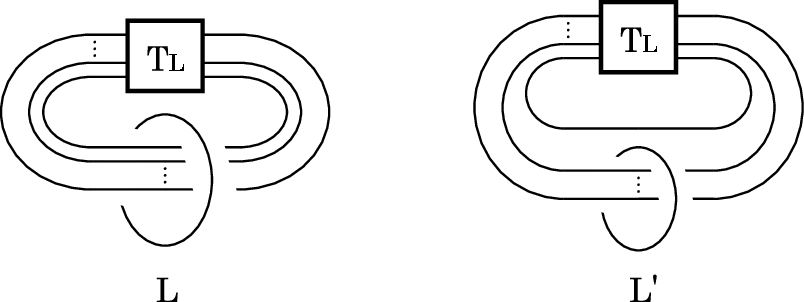}
\caption{Link $L$ and its associated trivial link $L'$, where $T_L$ is a tangle.}
\label{fig:associated_link}
\end{center}
\end{figure}

The diffeomorphism $\varphi_L: \partial X_L\to \partial X_{\widetilde{L}}$ is described in the upper side of Figure~\ref{fig:diffeo_dot_zero}. Since $L'$ is a trivial link, we can check that $\varphi_L$ is a composition of diffeomorphisms which extend to diffeomorphisms between 4-manifolds as shown in the figure. Hence the claim follows. Note that the dot-zero surgery corresponding to the Hopf link is a surgery along $D^4$, and that any self-diffeomorphism of $S^3$ extends to a self-diffeomorphism of $D^4$ as is well-known. 
\begin{figure}[h!]
\begin{center}
\includegraphics[width=4.4in]{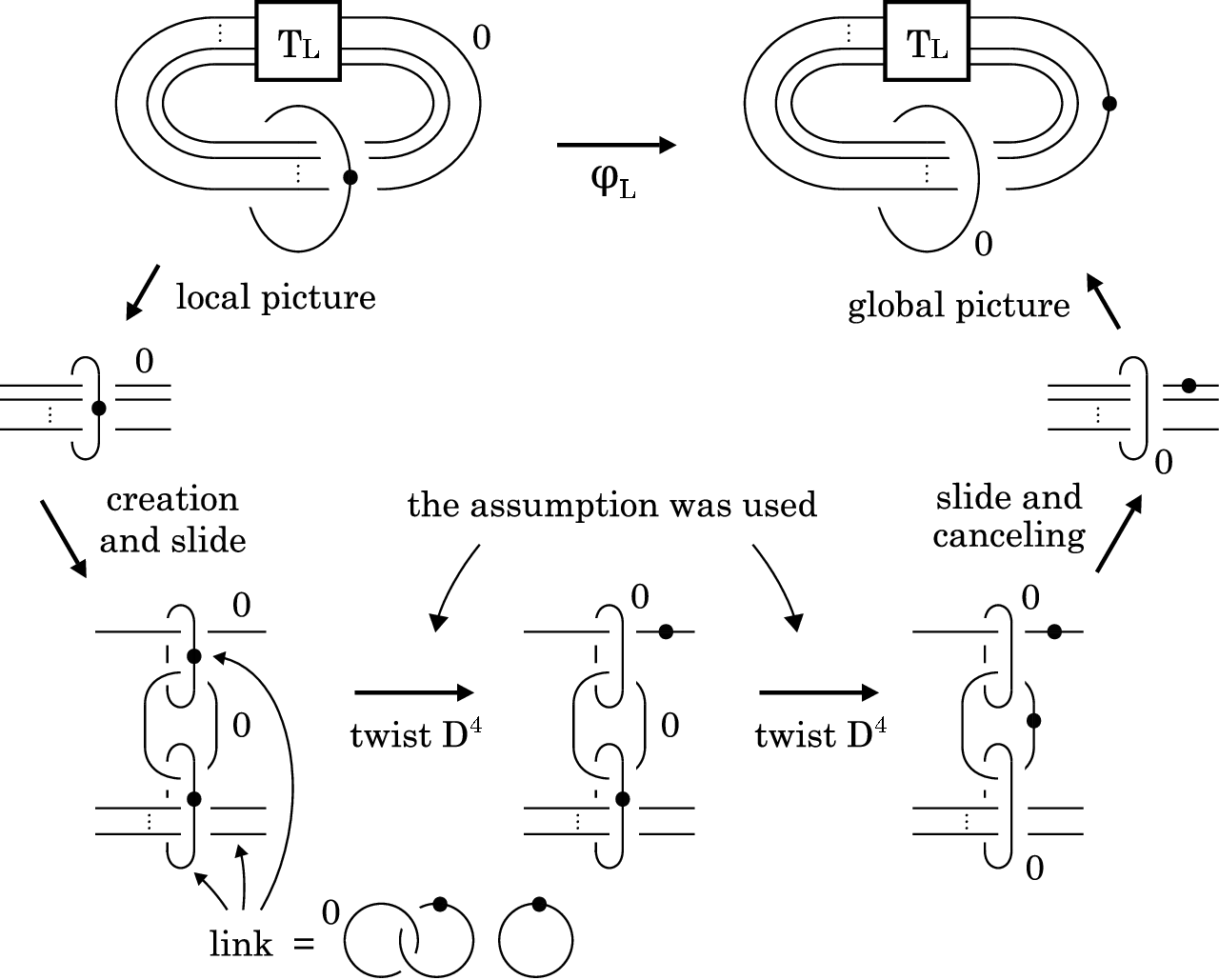}
\caption{Decomposition of the diffeomorphism $\varphi_L:\partial X_L\to \partial X_{\widetilde{L}}$}
\label{fig:diffeo_dot_zero}
\end{center}
\end{figure}
\end{proof}

Let us discuss an example coming from the new description of cork twists. Let $J_n$ be the link in $S^3$ given by the diagram of $V_n^*$ in Figure~\ref{fig:V_n_star}. Since $V_n^{*}$ $(n\geq 0)$ is diffeomorphic to the Stein 4-manifold $V_n$, and $(V_n, g_n)$ is a cork, one might expect that a dot-zero surgery along $(X_{J_n}, \varphi_{J_n})$ can alter smooth structures of 4-manifolds. %In fact, if a tangle $T_L$ of a link $L$ admits a Legendrian representative with $tb\geq 1$, then we can show that a dot-zero surgery along $(X_L, \varphi_L)$ can alter smooth structures (cf.\ \cite{AM, AY1}). 
However, by the above proposition, any dot-zero surgery along $(X_{J_n}, \varphi_{J_n})$ does not change smooth structures. Hence we obtain the following corollary. 
\begin{corollary}For each integer $n$, the diffeomorphism $\varphi_{J_n}:\partial X_{J_n}\to \partial X_{\widetilde{J}_n}$ extends to a diffeomorphism $X_{J_n}\to X_{\widetilde{J}_n}$. Furthermore, $X_{J_n}$ admits a Stein structure for $n\geq 0$. 
\end{corollary}
We remark that the Hopf link has been the only known such example to the best of the author's knowledge. Since $X_{J_n}=V_n^*$ is diffeomorphic to $V_n$, and $V_n$ is a cork for each $n\geq 0$, any $J_n$ $(n\geq 0)$ is not isotopic to the Hopf link. Note that the 4-ball is not a cork. These links $J_n$'s are probably mutually non-isotopic links, but we do not pursue this point here. 

As a special case, we discuss cork twists along Mazur type corks. Assume that a (ordered) link $L$ of two unknotted components in $S^3$ is symmetric, i.e., there exists an involution on $S^3$ which exchanges the components of $L$. Then the link $\widetilde{L}$ coincides with $L$ (up to order), and the symmetry induces a diffeomorphism $X_L\to X_{\widetilde{L}}$. Composing the boundary diffeomorphism $\varphi_L^{-1}: \partial X_{\widetilde{L}}\to \partial X_L$ to the restriction of this diffeomorphism, we obtain an involution $\tau_L:\partial X_L\to \partial X_L$. Furthermore, if the linking number of $L$ is one, then $X_L$ is contractible. Hence $(X_L, \tau_L)$ is a candidate of a cork in this case. When $(X_L, \tau_L)$ is a cork, then it is called of \textit{Mazur type}. It is well-known that a mild condition on the handlebody $X_L$ guarantees that $(X_L, \tau_L)$ is a Mazur type cork (\cite{AKa12, AM}). On the other hand, the above proposition immediately gives a simple sufficient condition that $(X_L, \tau_L)$ is not a cork. 

\begin{corollary}\label{cor:dot-zero:cork}Suppose that a symmetric link $L$ of two unknotted components in $S^3$ becomes a trivial link after a single crossing change between two components in a diagram of $L$. Then the involution $\tau_L: \partial X_L\to \partial X_{L}$ extends to a diffeomorphism $X_L\to X_{L}$. Consequently $(X_L, \tau_L)$ is not a cork. 
\end{corollary}
Applying the above corollary, we construct infinitely many symmetric links not yielding corks. 

\begin{proof}[Proof of Theorem~\ref{intro:thm:noncork}]
For an integer $n$, let $L_n$ be the link in $S^3$ given by Figure~\ref{fig:non-cork_L_n}, ignoring the dot and $0$. This link clearly consists of two unknotted components. We can easily check the isotopies of the link in Figure~\ref{fig:L_n_symmetry}, and the last diagram gives a symmetric link presentation of $L_n$. It follows from Corollary~\ref{cor:dot-zero:cork} that $(X_{L_n}, \tau_{L_n})$ is not a cork for any $n$. By using the skein relation along the $-n$ full twists of $L_n$, we can check that the HOMFLY-PT polynomials of $L_n$'s are pairwise distinct. Hence $L_n$'s are pairwise distinct links. 
\begin{figure}[h!]
\begin{center}
\includegraphics[width=1.3in]{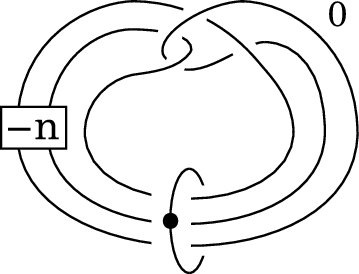}
\caption{$X_{L_n}$}
\label{fig:non-cork_L_n}
\end{center}
\end{figure}
\begin{figure}[h!]
\begin{center}
\includegraphics[width=3.2in]{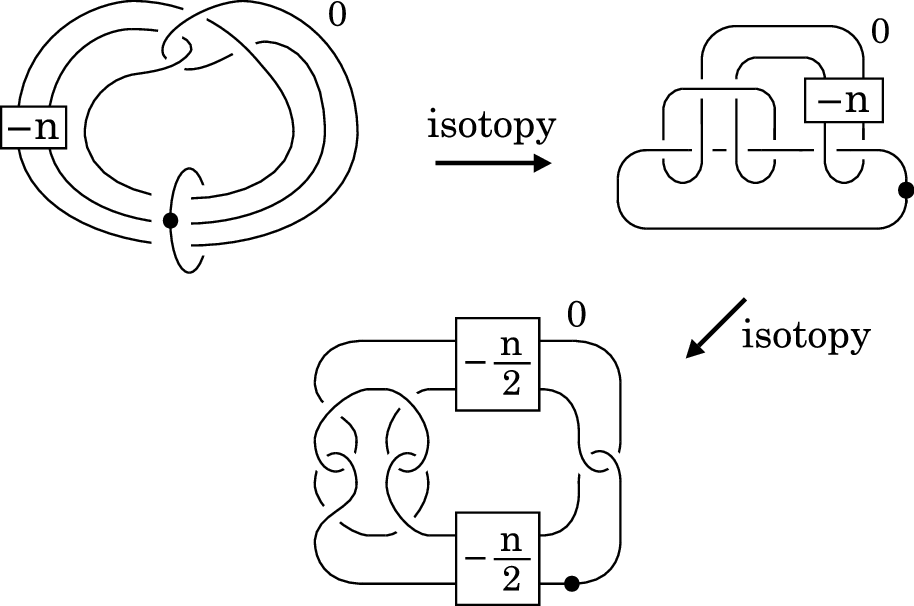}
\caption{Diagrams of $X_{L_n}$}
\label{fig:L_n_symmetry}
\end{center}
\end{figure}

We here alter the handlebody diagram of $X_{L_n}$, see Figure~\ref{fig:L_n_diffeo}. By repeating the handle moves in Figure~2.10 of \cite{A_book}, one can get the second diagram. A simple isotopy gives the third diagram, and the aforementioned handle moves give the last diagram of $X_{L_n}$. Corollary~1.7 in \cite{AKa14} and the last diagram thus tells that the Casson invariant $\lambda(\partial X_{L_n})$ of the boundary 3-manifold is equal to $-2$. Therefore each $X_{L_n}$ is not homeomorphic to $D^4$. Also, it easily follows from the last diagram that $X_{L_n}$ admits a Stein structure for $n\geq 2$. This completes the proof. 
\begin{figure}[h!]
\begin{center}
\includegraphics[width=3.4in]{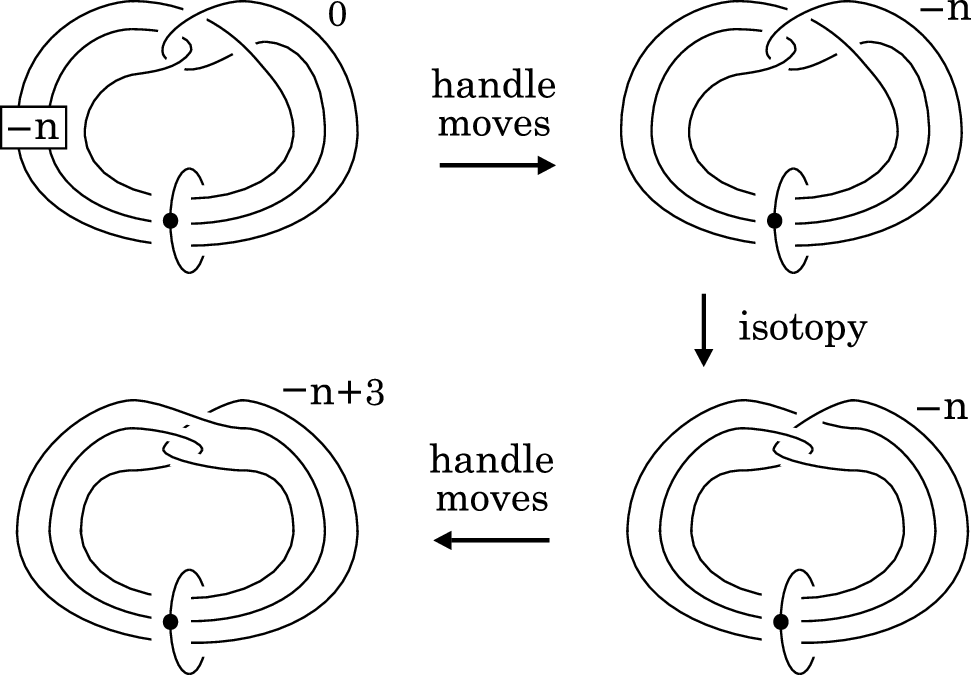}
\caption{Diagrams of $X_{L_n}$}
\label{fig:L_n_diffeo}
\end{center}
\end{figure}
\end{proof}

Using the new description of cork twists, we can easily prove Theorem~\ref{intro:thm:cork_noncork}. 
\begin{proof}[Proof of Theorem~\ref{intro:thm:cork_noncork}]
We use the notation of the above proof of Theorem~\ref{intro:thm:noncork}. One can easily check that the symmetric link $L_3$ is isotopic to the link in the diagram of $V_0^*$ in Figure~\ref{fig:V_n_star} (see also the right diagram of Figure~\ref{fig:pattern}). Therefore $X_{L_3}$ is diffeomorphic to $V_0$. Alternatively this can be directly checked from Figure~\ref{fig:L_n_diffeo}, since the last diagram $(n=3)$ is a diagram of $V_0$. Since $(V_0, g_0)$ is a Stein cork induced from a symmetric link, and $(X_{L_3}, \tau_{L_3})$ is not a cork, the claim follows. 
\end{proof}

\subsection{Hook surgery}
For a 2-component link $L$ in $S^3$ with one unknotted component, let $X_L$ denote the handlebody obtained by putting a dot on the unknotted component and a $0$ on top of the other component as shown in Figure~\ref{fig:hook}. This is a simple extension of the notation in the last subsection, and the other component is not necessarily an unknot in this case. For a knot $K$ in $S^3$, we present $K$ as in Figure~\ref{fig:tangle_smooth} using a tangle $T_K$, and let $\gamma^L_K$ be the knot in $\partial X_{L}$ shown in Figure~\ref{fig:hook_diffeo}. %we do not assume a symmetry of $L$.

\begin{figure}[h!]
\begin{center}
\includegraphics[width=1.2in]{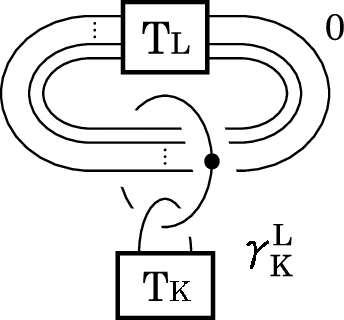}
\caption{The knot $\gamma^L_K$ in $\partial X_{L}$}
\label{fig:hook_diffeo}
\end{center}
\end{figure}

Now let $L_1$ and $L_2$ be two 2-component links with unknotted components. 
%, and assume further that the linking number of the two components of each $L_i$ is one. This assumption guarantees that each 4-manifold $X_{L_i}$ is contractible. 
Suppose that a diffeomorphism $\psi:\partial X_{L_1}\to \partial X_{L_2}$ satisfies the following conditions. 
\begin{itemize}
 \item $\psi$ maps the knot $\gamma^{L_1}_K$ to $\gamma^{L_2}_K$ for any knot $K$ in $S^3$. 
 \item $\psi$ extends to a homeomorphism $X_{L_1}\to X_{L_2}$.
\end{itemize}
Then we call each $X_{L_i}$ a \textit{hook} and call $(X_{L_1}, X_{L_2}: \psi)$ a \textit{hook pair}. For a smooth 4-manifold $Z$ containing $X_{L_1}$ as a submanifold, remove $X_{L_1}$ from $Z$ and glue $X_{L_2}$ via the gluing map $\psi$. We call this operation a \textit{hook surgery} along $(X_{L_1}, X_{L_2}: \psi)$. As seen from the proof of Theorem~\ref{intro:thm:satellite_knot}, a hook surgery is useful for constructing pairs of (non-concordant) knots and links representing exotic pairs of 4-manifolds. 

\begin{remark}\label{rem:hook}$(1)$ For a hook pair $(X_{L_1}, X_{L_2}: \psi)$, the diffeomorphism $\psi$ preserves the framing of $\gamma^L_{K}$ for any knot $K$. This can be seen by comparing the intersection forms of 4-manifolds $X_{L_i}$ with a 2-handle attached along $\gamma^{L_i}_K$, since the resulting 4-manifolds are homeomorphic.\smallskip\\
$(2)$ Assume that the linking number of the two components of each $L_i$ is one. Then any diffeomorphism $\partial X_{L_1}\to \partial X_{L_2}$ extends to a homeomorphism $X_{L_1}\to X_{L_2}$, since the both 4-manifolds are contractible. Therefore, if a diffeomorphism $\psi:\partial X_{L_1}\to \partial X_{L_2}$ maps the knot $\gamma^{L_1}_K$ to $\gamma^{L_2}_K$ for any knot $K$, then $(X_{L_1}, X_{L_2}: \psi)$ is a hook pair.  
\end{remark}

By Theorem~\ref{thm:boundary_V_n^*}, each $(V_n, V_n^*: g_n^*)$ is an example of a hook pair. Furthermore, Corollary~\ref{cor:new_description} tells that a hook surgery along $(V_n, V_n^*: g_n^*)$ is the same operation as a cork twist along $(V_n,g_n)$. Extending this example, we give a simple sufficient condition that a dot-zero surgery admits a hook surgery description. %Hence, constructing links in $S^3$ admitting a hook pair is a natural problem. %It is also interesting to characterize corks of Mazur type admitting compatible hook surgeries.

\begin{proposition}\label{prop:hook}
For an ordered link $L$ of two unknotted components in $S^3$, suppose that the component $K_1$ corresponding to the 1-handle of $X_{L}$ links the other component $K_2$ geometrically exactly once up to isotopy after untwisting $n$ full twists of two strands of $K_1$ as shown in Figure~\ref{fig:hook_prop} for an integer $n$. Then there exist an ordered 2-component link $L^*$ in $S^3$ with one unknotted component and a diffeomorphism $\varphi_L^*:\partial X_{L}\to \partial X_{L^*}$ such that the following conditions hold. 
\begin{itemize}
 \item $(X_L, X_{L^*}: \varphi_L^*)$ is a hook pair. 
 \item $\varphi_L^* \circ \varphi_{L}^{-1}:\partial X_{\widetilde{L}}\to \partial X_{L^*}$ extends to a diffeomorphism $X_{\widetilde{L}}\to X_{L^*}$. 
\end{itemize}
\begin{figure}[h!]
\begin{center}
\includegraphics[width=3.0in]{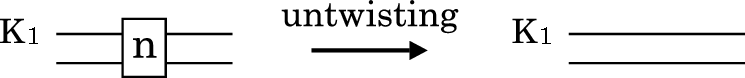}
\caption{Untwisting operation}
\label{fig:hook_prop}
\end{center}
\end{figure}
\end{proposition}

\begin{corollary}Let $L, L^*, \varphi_L^*$ be as above. Then, for any 4-manifold $Z$ containing $X_L$ as a submanifold, the 4-manifold obtained from $Z$ by the dot-zero surgery along $(X_L, \varphi_L)$ is diffeomorphic to the 4-manifold obtained from $Z$ by the hook surgery along $(X_L, X_{L^*}: \varphi_L^*)$. 
\end{corollary}
\begin{proof}[Proof of Proposition~\ref{prop:hook}]  We fix a knot $K$ in $S^3$. The diffeomorphism $\varphi_L:\partial X_L\to \partial X_{\widetilde{L}}$ clearly maps the knot $\gamma^L_K$ to the knot $\delta^L_K$ shown in Figure~\ref{fig:delta_K}. 
\begin{figure}[h!]
\begin{center}
\includegraphics[width=1.2in]{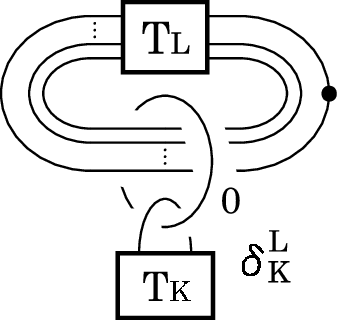}
\caption{The knot $\delta^L_K$ in $\partial X_{\widetilde{L}}$}
\label{fig:delta_K}
\end{center}
\end{figure}

We consider the local part of the handlebody diagram of $X_{\widetilde{L}}$ in the first diagram of Figure~\ref{fig:delta_K_move}, where the $n$ full twists are those in the assumption of the proposition. We introduce a canceling pair of 1- and 2-handles as shown in the second diagram, and let $J_1$ and $J_2$ be the knots corresponding to these 1- and 2-handles in the second diagram. By sliding the knots $\delta^{{L}}_K$ and $K_1$ over the 2-handle $J_2$, we obtain the third diagram. 
\begin{figure}[h!]
\begin{center}
\includegraphics[width=3.9in]{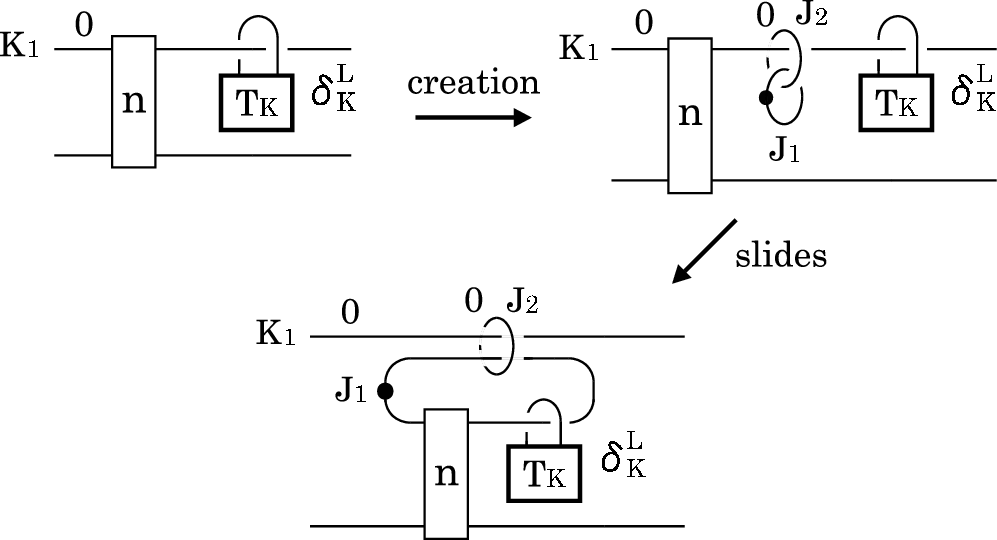}
\caption{Handle moves of $X_{\widetilde{L}}$}
\label{fig:delta_K_move}
\end{center}
\end{figure}

By the assumption of the proposition, we can isotope the resulting link $L\sqcup \delta^{{L}}_K$ so that $K_1$ links $K_2$ geometrically once, and that $J_1$ and $\delta^{{L}}_K$ do not link $K_2$. 
We then slide $J_2$ over the 2-handle $K_1$ so that $J_2$ does not link $K_2$ (and $\delta^{{L}}_K$). We can now eliminate the canceling handle pair $K_1$ and $K_2$ without changing the diffeomorphism type of $X_{\widetilde{L}}$.  

Let $L^*$ be the link in $S^3$ consisting of the resulting knots $J_1$ and $J_2$. Then the resulting handlebody diagram clearly coincides with that of $X_{L^*}$, and thus we obtain a diffeomorphism $\Phi:X_{\widetilde{L}}\to X_{L^*}$ which maps the knot $\delta^L_K$ to $\gamma^{L^*}_K$ on their boundaries. We define a diffeomorphism $\varphi_L^*:\partial X_{L}\to \partial X_{L^*}$ by $\varphi_L^*=\Phi|_{\partial X_{\widetilde{L}}}\circ \varphi_L$. Clearly $\varphi_L^*$ maps the knot $\gamma^{L}_K$ to $\gamma^{L^*}_K$ for any $K$, and the linking numbers of $L$ and $L^*$ are both one. Hence Remark~\ref{rem:hook} tells that $(X_L, X_{L^*}: \varphi_L^*)$ is a hook pair. The definition of $\varphi_L^*$ guarantees that $\varphi_L^*\circ \varphi_{L}^{-1}$ extends to a diffeomorphism $X_{\widetilde{L}}\to X_{L^*}$. This completes the proof. 
\end{proof}

The above proposition provides many corks admitting hook surgery descriptions. A natural question is whether all Mazur type corks admit hook surgery descriptions. We remark that, if the answer to Question~7.6 in \cite{AY1} is negative, then the cork $(W_n, f_n)$ $(n\geq 2)$ obtained in \cite{AY1} admits no hook surgery description. 
%%%%%%%%%%%%%%%%%%
%%%%%%%%%%%%%%%%%%%
%%%%%%%%%%%%%%
%%%%%%%%%%%%
%\medskip\\
%%%%%%%%%%%%%%%%%
\subsection*{Acknowledgements} The application to knot concordance was motivated by Tetsuya Abe's nice talk on the paper \cite{AT} at the workshop ``Differential Topology '15: Past, Present, and Future'' held at Kyoto University. The author is grateful to the organizers for making the enjoyable workshop. He thanks Tetsuya Abe and Masakazu Teragaito for their many helpful comments on the first draft. He also thanks Adam Simon Levine, Arunima Ray, Yuichi Yamada, and a referee for their useful comments. 

%%%%%%%%%%%%%%%%
%%%%%%%%%%%%%%%%%%%%%%%%%%%%%%%%%%%%%%%%%%%%%%%%%%%%%%%%%
%%%%%%%%%%%%%%%%%%%%%%%%%%%%%%%%%%%%%%%%%%%%%%%%%%%%%

\end{document}